\documentclass{amsart}

\usepackage{amsmath,amssymb,amscd}
\usepackage{graphicx}
\usepackage[margin=1.1in]{geometry}
\usepackage{mathdots,slashed,mathabx}
\usepackage{blkarray}
\usepackage[hidelinks]{hyperref}
\usepackage[all]{xy}
\usepackage{float}
\input{quivers_ext.sty}

\newtheorem{theorem}{Theorem}[section]
\newtheorem{lemma}[theorem]{Lemma}
\newtheorem{corollary}[theorem]{Corollary}
\newtheorem{proposition}[theorem]{Proposition}

\theoremstyle{definition}
\newtheorem{definition}[theorem]{Definition}
\newtheorem{example}[theorem]{Example}

\numberwithin{equation}{section}

\DeclareMathOperator{\Coker}{Coker}

\DeclareMathOperator{\GL}{GL}
\DeclareMathOperator{\Gr}{Gr}
\DeclareMathOperator{\Hom}{Hom}

\DeclareMathOperator{\PHom}{PHom}

\DeclareMathOperator{\Rep}{Rep}

\DeclareMathOperator{\SI}{SI}
\DeclareMathOperator{\SL}{SL}

\DeclareMathOperator{\Spec}{Spec}

\newcommand{\mc}[1]{\mathcal{#1}}
\newcommand{\mb}[1]{\mathbb{#1}}

\renewcommand{\b}[1]{\mathbf{#1}}
\newcommand{\bs}[1]{\boldsymbol{#1}}

\renewcommand{\d}{{\sf d}}
\newcommand{\e}{{\sf e}}
\newcommand{\g}{{\sf g}}

\renewcommand{\t}{{\sf t}}

\newcommand{\fr}[1]{\framebox[1.2\width]{{$#1$}}}

\newcommand{\uca}{\mc U}

\newcommand{\ul}[1]{\underline{#1}}

\newcommand{\ucaloc}{\mc U^{\mathrm{loc}}}
\newcommand{\ucarel}[1]{\mc U^{\mathrm{loc}}_{#1}}
\newcommand{\Gloc}{G^{\mathrm{loc}}}

\begin{document}
\title{Extending Upper Cluster Algebras}
\author{Jiarui Fei}
\address{School of Mathematical Sciences, Shanghai Jiao Tong University, Shanghai 200240, China}
\email{jiarui@sjtu.edu.cn}
\author{Jerzy Weyman}
\address{Department of Mathematics, University of Connecticut, Storrs, CT}
\email{jerzy.weyman@uconn.edu}

\date{}
\thanks{The first author was partially supported by National Natural Science Foundation of China (No. 12131015 and No. 12571038);
The second author was partially supported by NSF grant DMS-1400740.}

\begin{abstract}
We study how adjoining a vertex affects an upper cluster algebra and its generic basis after the frozen variables are inverted. Adjoining a frozen sink or source preserves the basis when its coefficient column is rationally balanced. For a mutable source, we give a saturation criterion for the upper cluster algebra to be generated by the adjacent cluster variable over the relative frozen algebra. A triangular formula for generic characters then proves the basis theorem. Valuations along the frozen divisors determine the ordinary generic basis. The support cone is rational convex polyhedral when each frozen vertex can be made a source after the other frozen vertices are deleted. We also treat amalgamation along frozen vertices and identify the upper cluster algebras of complete $m$-tuple flags. For glued hive quivers, an explicit reverse-deletion construction gives polyhedral generic bases.
\end{abstract}

\maketitle

\section*{Introduction}

Let $(\Delta,W)$ be an ice quiver with potential. In \cite{Fs2} we studied projections obtained by deleting a vertex after localizing the corresponding cluster variable. Here we ask when adjoining a vertex preserves an upper cluster algebra and its generic-character basis.

We first invert all frozen variables, then recover the ordinary algebra by requiring regularity along the frozen divisors. Related results concern covering pairs \cite{Mu}, generic bases under injective reachability \cite{Qin}, and theta bases for partial compactifications \cite{GHKK,CMMM}.

The generic character of $(\Delta,W)$ extends from the $\mu$-supported set $G(\Delta,W)$ to
\[ \Gloc(\Delta,W)=G(\Delta,W)+\mathbb Z^{\Delta_\nu}. \]
We call an IQP a \emph{lattice model} if its localized generic characters form a basis indexed by $\mathbb Z^{\Delta_0}$. For a Jacobi-finite IQP, every weight is obtained by shifting a supported weight in the frozen directions. The basis property is mutation invariant.

For a frozen sink or source, rational balancing identifies the extended algebra with a cyclic homogenization of the original algebra. For a mutable source $e$, let $A_e$ be the relative algebra in which $e$ is frozen but $x_e$ is not inverted. We study the equality
\[ \ucaloc(\ul\Delta)=A_e[x_e']. \]
Under the intersection hypothesis stated in Section~\ref{S:saturation}, this equality is equivalent to saturation of the special fiber $A_e/(x_e,\rho_e)$. A filtration by powers of $x_e'$ and a triangular formula for generic characters give the extended basis. A mutable sink is reduced to this case by one mutation.

\begin{theorem}[Localized extension]\label{T:intro-local}
Let $(\Delta,W)$ be a lattice model, and let $(\ul\Delta,\ul W)$ be obtained by finitely many one-vertex extensions and admissible mutations. Suppose that each new vertex can be made a sink or source by mutating only vertices already present. Assume rational balancing whenever the new vertex remains frozen. When the new vertex is mutable, first make it a source, and assume rational balancing and the equality $\ucaloc(\ul\Delta)=A_e[x_e']$. Then $(\ul\Delta,\ul W)$ is a lattice model.
\end{theorem}

The rank assumptions in Corollary~\ref{C:iterated-fullrank} imply the required balancing and algebra equalities. We also prove that lattice models are preserved by amalgamation along frozen vertices when the mutable parts are disjoint.

At a frozen source, the boundary valuation is the $g$-coordinate; at a sink it is the dual $g$-coordinate. We prove valuative independence by shifting presentation weights and passing to the opposite QP. Saturating at the other frozen vertices gives the support condition at each boundary, and intersecting these conditions gives $\mu$-support.

\begin{theorem}[Boundary compactification]\label{T:intro-boundary}
Let $(\Delta,W)$ be a lattice model, and suppose that every frozen vertex can be made a source or sink after deleting the other frozen vertices, by an admissible mutation sequence avoiding that vertex.  Let $\vartheta_v$ be the resulting boundary function.  Then
\[ \uca(\Delta) = \bigoplus_{\substack{\g\in\mathbb Z^{\Delta_0}\\ \vartheta_v(\g)\ge0\ \forall v\in\Delta_\nu}} k\,C_W^{\mathrm{loc}}(\g), \]
and the indexing set is $G(\Delta,W)$. Suppose, in addition, that after deleting the other frozen vertices, each frozen vertex can be made a source by an admissible sequence of mutable mutations. Then this set is the set of integral points of a rational convex polyhedral cone.
\end{theorem}

For hive quivers, we use the models in \cite{Fs1,Fs2}. An explicit reverse-deletion construction gives lattice models for the glued hives, and boundary mutation sequences give their polyhedral generic bases. Gluing along generic pairs of flags identifies the localized rings, and two triangulations with no common diagonal give the ordinary ring equality.

Sections~\ref{S:localized-foundations}--\ref{S:iterated-localized} prove the extension theorems. Sections~\ref{S:amalgamation}--\ref{S:boundary} treat amalgamation and boundary compactification. Section~\ref{S:applications} applies these results to flag semi-invariants. The appendices contain the hive calculations and the gluing argument for generic pairs.

\paragraph{Conventions.}
The base field is $k=\mathbb C$. All vectors are row vectors, and all modules are finite-dimensional right modules. Throughout Sections~\ref{S:localized-foundations}--\ref{S:boundary}, every IQP is reduced, Jacobi-finite, and nondegenerate. A mutation sequence is \emph{admissible} if each reduced DWZ mutation is defined and all intermediate IQPs have these properties. Arrows between frozen vertices are retained in the completed path algebra and the potential but omitted from the exchange matrix. Restriction to a full sub-IQP applies to both the quiver and the potential.

\section{Coefficient-localized upper cluster algebras}\label{S:localized-foundations}

\subsection{Seeds and coefficient localization}

Let $(\Delta,\b x)$ be an ice seed of geometric type. The mutable and frozen vertex sets are denoted by $\Delta_\mu$ and $\Delta_\nu$. The extended cluster is $\b x=\{x_v\}_{v\in\Delta_0}$, and the exchange matrix is
\[ B(\Delta)=(b_{uv})_{u\in\Delta_\mu,\,v\in\Delta_0}, \qquad b_{uv}=\#(u\to v)-\#(v\to u). \]
We use seed mutation \cite{FZ1}, geometric-type coefficients \cite{FZ4}, and the upper-bound conventions of \cite{BFZ}. Arrows between frozen vertices are omitted from $B(\Delta)$ but retained in the IQP.

For a seed $(\Delta',\b x')$ mutation-equivalent to $(\Delta,\b x)$, set
\[ \mc L_{\Delta}(\b x') =k[x_v\mid v\in\Delta_\nu] [x_u'{}^{\pm1}\mid u\in\Delta_\mu] \]
and
\[ \mc L_{\Delta}^{\mathrm{loc}}(\b x') =k[x_v^{\pm1}\mid v\in\Delta_\nu] [x_u'{}^{\pm1}\mid u\in\Delta_\mu]. \]
The ordinary and coefficient-localized upper cluster algebras are
\[ \uca(\Delta,\b x)=\bigcap_{(\Delta',\b x')\sim(\Delta,\b x)} \mc L_\Delta(\b x'), \qquad \ucaloc(\Delta,\b x)=\bigcap_{(\Delta',\b x')\sim(\Delta,\b x)} \mc L_\Delta^{\mathrm{loc}}(\b x'). \]
Both intersections are taken in the common ambient field.

For $F\subseteq\Delta_\nu$, write
\[ \uca_F(\Delta)=\uca(\Delta)[x_f^{-1}\mid f\in F]. \]
For a frozen vertex $v$, let $\nu_v$ be the initial $x_v$-adic valuation, with every other initial variable of value zero.

\begin{lemma}\label{L:boundary-intersection}
For every ice seed, we have
\begin{equation}\label{eq:boundary-intersection} \uca(\Delta) =\{z\in\ucaloc(\Delta)\mid\nu_v(z)\ge0 \text{ for every }v\in\Delta_\nu\}. \end{equation}
More generally, for $F\subseteq\Delta_\nu$,
\begin{align}
 \uca_F(\Delta)
 &=\{z\in\ucaloc(\Delta)\mid\nu_v(z)\ge0
                         \text{ for }v\in\Delta_\nu\setminus F\}
                         \label{eq:partial-boundary-intersection}\\
 &=\bigcap_{\b x'\sim\b x}
   k[x_f^{\pm1}\mid f\in F]
    [x_v\mid v\in\Delta_\nu\setminus F]
    [x_u'{}^{\pm1}\mid u\in\Delta_\mu].\nonumber
\end{align}
In particular $\ucaloc(\Delta)=\uca(\Delta)[x_f^{-1}\mid f\in\Delta_\nu]$.
\end{lemma}

\begin{proof}
We show that $\nu_v$ is the Gauss valuation in every cluster chart. Suppose the other cluster variables have value zero and algebraically independent residues. Under mutation at $u$, at most one exchange monomial contains $x_v$, so the residue $\bar\rho_u$ of the exchange polynomial is a nonzero monomial or sum of two monomials. The new variable has value zero, and the residue transformation
$\bar x_u\mapsto\bar\rho_u/\bar x_u$
is birational, since $\bar\rho_u$ is independent of $\bar x_u$. Algebraic independence of the residues is preserved. Induction shows that $\nu_v$ is the least $x_v$-exponent in every chart.

In each Laurent chart, membership in the partially localized ring is equivalent to nonnegativity of the valuations at the frozen variables that have not been inverted. Intersecting over all charts gives the right-hand descriptions in~\eqref{eq:boundary-intersection}--\eqref{eq:partial-boundary-intersection}. Multiplying by a sufficiently large monomial in the variables of $F$ clears the remaining negative frozen valuations and gives an element of $\uca(\Delta)$. Hence the intersection equals $\uca_F(\Delta)$.
\end{proof}

\begin{theorem}[Upper bound, {\cite[Theorems~3.11 and~4.1]{GSV2}}]\label{T:upper-bound}
If $B(\Delta)$ has full row rank, then both $\uca(\Delta)$ and $\ucaloc(\Delta)$ are obtained by intersecting the initial Laurent ring with the Laurent rings reached by a single mutation.
\end{theorem}

\begin{proposition}[Freezing after localization, {\cite[Lemma~2.4]{Fs2}}]\label{P:freezing-localization}
Let $E\subseteq\Delta_\mu$, let $\Delta^E$ be obtained by freezing the vertices in $E$, and assume that $B(\Delta)$ has full row rank. Then $\uca(\Delta^E)\subseteq\uca(\Delta)$, and
\[ \uca(\Delta)[x_e^{-1}\mid e\in E] = \uca(\Delta^E)[x_e^{-1}\mid e\in E] \]
inside the common fraction field.
\end{proposition}

\begin{lemma}[Primality of frozen variables]\label{L:frozen-prime}
Let $v$ be frozen, and let $E\subseteq\Delta_\nu\setminus\{v\}$. Put
\[ A_E=\uca(\Delta,\b x)[x_w^{-1}\mid w\in E], \]
and let
\[ L_E=k[x_w^{\pm1}\mid w\in E] [x_z\mid z\in\Delta_\nu\setminus E] [x_u^{\pm1}\mid u\in\Delta_\mu] \]
be the corresponding initial Laurent ring. Then
\[ A_E\cap x_vL_E=x_vA_E. \]
In particular, $x_v$ is prime in $A_E$.
\end{lemma}

\begin{proof}
By Lemma~\ref{L:boundary-intersection}, $A_E$ is the intersection of the corresponding Laurent rings, and $\nu_v$ is the least $x_v$-exponent in each chart. If $z\in A_E\cap x_vL_E$, then $z/x_v$ belongs to every such Laurent ring: its noninverted frozen valuations remain nonnegative. Hence $z/x_v\in A_E$. Thus the contraction of the prime ideal $x_vL_E$ is $x_vA_E$, which is therefore prime.
\end{proof}

\subsection{Generic characters and coefficient saturation}

Let $(\Delta,W)$ be a Jacobi-finite, nondegenerate IQP, and put $J=J(\Delta,W)$. A decorated representation determines a two-term projective presentation. Our convention is that the presentation
$P(\beta_1)\longrightarrow P(\beta_0)$
has weight (or $g$-vector) $\beta_1-\beta_0$. Thus, for $\g\in\mathbb Z^{\Delta_0}$, we set
\[ \PHom_J(\g)=\Hom_J(P([\g]_+),P([-\g]_+)). \]
Write $\Coker(\g)$ for a general cokernel of weight $\g$. Dimensions and Grassmannian Euler characteristics are taken on common dense open subsets of the presentation space. We call $\g$ \emph{$\mu$-supported} if this cokernel vanishes at every frozen vertex, and write $G(\Delta,W)$ for the set of such weights.

For $\g\in G(\Delta,W)$, the generic character is
\begin{equation}\label{eq:generic-character} C_W(\g)=\b x^{\g} \sum_{\d}\chi\bigl(\Gr^{\d}(\Coker(\g))\bigr)\b y^{\d}, \qquad y_u=\b x^{-b_u}. \end{equation}
Here $\Gr^{\d}$ parametrizes quotient representations, and $C_W(\g)\in\uca(\Delta)$ by \cite[Theorems~1.1 and~1.3]{P}.

\begin{lemma}\label{L:idempotent-quotient-shifts}
Let $A$ be a finite-dimensional basic algebra with vertex set $I$, let $E\subseteq I$, and put $A_E=A/Ae_EA$, where $e_E=\sum_{i\in E}e_i$. Use the convention for presentation weights given above. For any $\g\in\mathbb Z^I$ and all sufficiently large $\bs a\in\mathbb Z_{\ge0}^{E}$, a general cokernel of weight $\g+\bs a$ vanishes on $E$ and is a general $A_E$-cokernel of weight $\g|_{I\setminus E}$, regarded as an $A$-module.

Write $\check F^A_\g(\b t)$ for the generic quotient $F$-polynomial in independent variables $\b t$. Then
\begin{equation}\label{eq:idempotent-F-face} \check F^A_{\g+\bs a}(\b t) =\check F^{A_E}_{\g|_{I\setminus E}}(\b t) =\sum_{\d|_E=0} \chi\bigl(\Gr^{\d}(\Coker_A(\g))\bigr)\b t^{\d}. \end{equation}
The variables and dimension vectors for $A_E$ are extended by zero on $E$. For a finite collection of weights, one sufficiently large $\bs a$ works for all of them.
\end{lemma}

\begin{proof}
Reduction induces surjections $e_jAe_i\to e_jA_Ee_i$ on projective Hom spaces, so a general presentation descends to a general presentation of weight $\g|_{I\setminus E}$.

For sufficiently large $\bs a$, the target $P_0$ has no summand indexed by $E$ and does not change when $\bs a$ is increased further. The source contains arbitrarily many copies of each $P_e$, $e\in E$. Since $\Hom_A(P_e,P_0)=P_0(e)$, enough copies can be mapped onto each component $P_0(e)$. Surjectivity on all these components is a nonempty open condition. Intersect this open set with the inverse image of the generic locus in the quotient presentation space.

The quotients of a module $M$ that vanish on $E$ are exactly the quotients of $M/Me_EA$, and their Grassmannians are the same. Cokernels commute with this quotient. Apply the surjections between presentation spaces to the shifted and unshifted weights. Intersecting the generic loci for the finitely many quotient dimensions proves~\eqref{eq:idempotent-F-face}. Taking a common coordinatewise bound gives one shift for any finite collection of weights.
\end{proof}

\begin{lemma}\label{L:frozen-shifts}
If $\g\in G(\Delta,W)$, then $\g(v)\ge0$ at every frozen vertex. For every $\bs a\in\mathbb Z_{\ge0}^{\Delta_\nu}$,
\[ \g+\bs a\in G(\Delta,W),\qquad C_W(\g+\bs a)=\b x^{\bs a}C_W(\g). \]
\end{lemma}

\begin{proof}
If $\g(v)<0$, the top of a target summand $P_v$ cannot be hit by source projectives of other types. Thus the cokernel is nonzero at $v$. For a presentation $f:P_1\to P_0$ with supported cokernel $M$, projectivity and $\Hom_J(P_v,M)=0$ imply that every map $P_v\to P_0$ factors through $f$. A change of basis in the source makes the maps from the added frozen summands zero, so the cokernel is unchanged. The projection to the original presentation space is surjective, so this holds for a general presentation. Apply~\eqref{eq:generic-character}.
\end{proof}

Define
\[ \Gloc(\Delta,W)=G(\Delta,W)+\mathbb Z^{\Delta_\nu}. \]
For $\g\in\Gloc(\Delta,W)$ choose $\bs a\in\mathbb Z_{\ge0}^{\Delta_\nu}$ with $\g+\bs a\in G(\Delta,W)$, and set
\begin{equation}\label{eq:localized-character} C_W^{\mathrm{loc}}(\g)=\b x^{-\bs a}C_W(\g+\bs a). \end{equation}
This is independent of the choice by Lemma~\ref{L:frozen-shifts}: compare any two shifts with a common larger shift.

\begin{corollary}\label{L:mutable-quotient-characters}
Put $J_\mu=J/Je_{\Delta_\nu}J$. Then
\begin{equation}\label{eq:full-coefficient-saturation} \Gloc(\Delta,W)=\mathbb Z^{\Delta_0}. \end{equation}
If $\gamma=\g|_{\Delta_\mu}$ and $M_\gamma$ is a general cokernel in $\PHom_{J_\mu}(\gamma)$, then
\begin{equation}\label{eq:mutable-quotient-character} C_W^{\mathrm{loc}}(\g)=\b x^\g\check F_{M_\gamma}(\b y). \end{equation}
Thus the generic dimension vector and quotient $F$-polynomial depend only on $\gamma$, not on the frozen coordinates of $\g$.
\end{corollary}

\begin{proof}
Apply Lemma~\ref{L:idempotent-quotient-shifts} with $A=J$ and $E=\Delta_\nu$, and then use~\eqref{eq:localized-character}. Any supported shift can be compared to a sufficiently large one by Lemma~\ref{L:frozen-shifts}.
\end{proof}

\begin{definition}\label{D:cluster-model}\label{D:lattice-model}
The IQP $(\Delta,W)$ is a \emph{cluster model} if
$\{C_W(\g)\mid\g\in G(\Delta,W)\}$
is a basis of $\uca(\Delta)$. It is a \emph{polyhedral cluster model} if, in addition,
\[ G(\Delta,W)=P\cap\mathbb Z^{\Delta_0} \]
for a rational polyhedron $P\subseteq\mathbb R^{\Delta_0}$.
It is a \emph{lattice model} if
$\{C_W^{\mathrm{loc}}(\g)\mid\g\in\mathbb Z^{\Delta_0}\}$
is a basis of $\ucaloc(\Delta)$.
\end{definition}

Let $M_\g$ be the general $J_\mu$-module in~\eqref{eq:mutable-quotient-character}, regarded as a $J$-module, and let $\b m_\g$ be its dimension vector on the mutable vertices. Define the coweight
\[ \g^\vee=\g-\b m_\g B(\Delta). \]
For submodules, we use negative exponents:
\[ F_M(\b y)=\sum_{\d}\chi(\Gr_{\d}(M))\b y^{-\d}, \qquad \check F_M(\b y)=\sum_{\d}\chi(\Gr^{\d}(M))\b y^{\d}. \]

\begin{lemma}\label{L:bipointed-character}\label{L:opposite-generic-duality}
Let $(\Delta^{\mathrm{op}},W^{\mathrm{op}})$ be the opposite IQP, with the same cluster variables. Then
\begin{equation}\label{eq:localized-bipointed} C_W^{\mathrm{loc}}(\g) =\b x^\g\check F_{M_\g}(\b y) =\b x^{\g^\vee}F_{M_\g}(\b y), \end{equation}
and
\begin{equation}\label{eq:opposite-generic-character} C_{W^{\mathrm{op}}}^{\mathrm{loc}}(\g^\vee) =C_W^{\mathrm{loc}}(\g). \end{equation}
The map $\g\mapsto\g^\vee$ is a bijection, with inverse given by the same construction for the opposite IQP. The two coefficient-localized upper cluster algebras are equal. A localized generic basis for either IQP is therefore also a basis for the other, with weights related by this bijection.
\end{lemma}

\begin{proof}
Passing from submodules to quotients gives $\check F_M(\b y)=\b y^{\dim M}F_M(\b y)$, proving~\eqref{eq:localized-bipointed}. The mutable quotient is the Jacobian algebra of the full mutable sub-QP. Indeed, the cyclic derivatives at arrows incident with a deleted vertex vanish in the idempotent quotient; the remaining derivatives are those of the restricted potential.

Apply the projective--injective genericity theorem \cite[Theorem~3.11(1)]{Fg} to this Jacobi-finite mutable QP. In that reference $\delta=-\gamma$, and the stable injective weight is
$\check\delta=\delta+\b m_\g B_\mu=-\g^\vee|_{\Delta_\mu}$.
Dualizing the general injective presentation shows that $D M_\g$ is a general cokernel for the opposite mutable quotient of weight $\g^\vee|_{\Delta_\mu}$. Here the weights include the decorations; see also \cite[Theorem~3.3]{Fc1}.

Since $B^{\mathrm{op}}=-B$ and $\b y^{\mathrm{op}}=\b y^{-1}$,
\[ \check F_{D M_\g}(\b y^{-1}) =\b y^{-\b m_\g}\check F_{M_\g}(\b y). \]
Formula~\eqref{eq:mutable-quotient-character} proves~\eqref{eq:opposite-generic-character}. Applying the construction twice gives
$\g^\vee-\b m_\g(-B)=\g$, so it is bijective. Replacing $B$ by $-B$ interchanges the two monomials in every exchange relation and commutes with matrix mutation; hence the two upper cluster algebras are equal.
\end{proof}

For a frozen vertex $v$ put
\[ G_{\widehat v} :=G(\Delta,W)+\mathbb Z^{\Delta_\nu\setminus\{v\}}, \qquad R_v:=\uca(\Delta)[x_f^{-1}\mid f\in\Delta_\nu\setminus\{v\}]. \]
Thus $\g\in G_{\widehat v}$ if a nonnegative shift supported away from $v$ makes it $\mu$-supported.

\begin{corollary}\label{L:rankfree-one-sided}
Assume that the localized generic characters form a basis of $\ucaloc(\Delta)$. Let $v$ be frozen and suppose that, after deleting the other frozen vertices, $v$ is a source or sink. Put
\[
 \lambda_v(\g)=
 \begin{cases}
  \g(v),&v\text{ is a source},\\[1mm]
  \g^\vee(v),&v\text{ is a sink}.
 \end{cases}
\]
Then
\begin{equation}\label{eq:rankfree-one-sided-value} \nu_v\bigl(C_W^{\mathrm{loc}}(\g)\bigr)=\lambda_v(\g), \end{equation}
and, for every nonzero finite sum,
\begin{equation}\label{eq:rankfree-one-sided-adapted} \nu_v\!\left(\sum_\g c_\g C_W^{\mathrm{loc}}(\g)\right) =\min_{c_\g\ne0}\lambda_v(\g). \end{equation}
Moreover,
\begin{equation}\label{eq:one-sided-relative-support} G_{\widehat v}=\{\g\in\mathbb Z^{\Delta_0}\mid\lambda_v(\g)\ge0\}, \end{equation}
and
\begin{equation}\label{eq:one-sided-relative-basis} R_v=\{z\in\ucaloc(\Delta)\mid\nu_v(z)\ge0\} =\bigoplus_{\g\in G_{\widehat v}}k\,C_W^{\mathrm{loc}}(\g). \end{equation}
No rank assumption on $B(\Delta)$ is needed. If $v$ is a source in the full quiver, the unshifted general presentation also vanishes at $v$ exactly when $\g(v)\ge0$.
\end{corollary}

\begin{proof}
Write $b_\g=C_W^{\mathrm{loc}}(\g)$. First suppose that $v$ is a source after the other frozen vertices are deleted, and let $E_v$ be its mutable out-neighbors. A term indexed by a quotient dimension vector $\d$ has $x_v$-exponent
\[ \g(v)+\sum_{u\in E_v}m_{vu}\d(u)\ge\g(v). \]
For a Laurent polynomial $z$, let $[z]_{v=r}$ be the sum of its terms with $x_v$-exponent $r$, including the factor $x_v^r$. Apply Lemma~\ref{L:idempotent-quotient-shifts} to $J_\mu$ and $E_v$. For any finite collection of weights, one can choose a common sufficiently large $\bs a\in\mathbb Z_{\ge0}^{E_v}$ such that
\[ \b x^{\bs a}[b_\g]_{v=\g(v)}=b_{\g+\bs a}. \]
These are precisely the quotients vanishing on $E_v$, whose generating series is given by the shifted presentation in the lemma.

The right-hand side is nonzero by the assumed basis property, proving~\eqref{eq:rankfree-one-sided-value} at a source. For a nonzero finite sum let $r$ be the least $\g(v)$ occurring. Its homogeneous part of exponent $r$, multiplied by a common monomial $\b x^{\bs a}$, is
$\sum_{\g(v)=r}c_\g b_{\g+\bs a}$.
The shifted weights are distinct, so this sum is nonzero by the same basis hypothesis. This proves~\eqref{eq:rankfree-one-sided-adapted}. At a sink, use Lemma~\ref{L:opposite-generic-duality} and apply the source result to the opposite IQP.

To prove the support assertion, put
$J^{[v]}=J/Je_{\Delta_\nu\setminus\{v\}}J$.
By Lemma~\ref{L:idempotent-quotient-shifts}, a weight $\g$ belongs to $G_{\widehat v}$ if and only if a general $J^{[v]}$-cokernel of weight $\g|_{\Delta_\mu\cup\{v\}}$ vanishes at $v$. A supported presentation remains supported after passage to the quotient, and sufficiently large shifts at the deleted vertices give a general presentation over that quotient.

At a source in this quotient, no projective other than $P_v$ has a component at $v$, so that cokernel has dimension $[-\g(v)]_+$ there. If $v$ is a source in the full quiver, the same computation over $J$ proves the last sentence. At a sink, write
\[ d_v(\g)=\sum_{u\in\Delta_\mu}m_{uv}\b m_\g(u). \]
No cycle in $J^{[v]}$ passes through $v$. Decomposing paths according to their last arrow shows that the cokernel of the presentation using the old projectives has dimension $d_v(\g)$ at $v$. Since $P_v=S_v$, copies of $P_v$ added to the source give arbitrary vectors in this component. Copies added to the target cannot be hit by the old projectives. The general cokernel at $v$ therefore has dimension
$\max\{d_v(\g)-\g(v),0\}$.
But $\g^\vee(v)=\g(v)-d_v(\g)$, proving~\eqref{eq:one-sided-relative-support}.

For $\g\in G_{\widehat v}$, a supported shift away from $v$ gives $b_\g\in R_v$. Every element of $R_v$ has nonnegative $v$-valuation. Expanding in the localized basis and applying~\eqref{eq:rankfree-one-sided-adapted} proves~\eqref{eq:one-sided-relative-basis}.
\end{proof}

\begin{lemma}\label{L:frozen-corner-saturation}
Let $A$ be a finite-dimensional basic algebra with vertex set $I=M\sqcup F$, and let $G_A$ be the set of presentation weights whose general cokernels vanish on $F$. For a fixed $\gamma\in\mathbb Z^M$, there is a vector $\b r(\gamma)\in\mathbb Z_{\ge0}^F$ such that
\begin{equation}\label{eq:frozen-fiber-orthant} (\gamma,\eta)\in G_A \quad\Longleftrightarrow\quad \eta(v)\ge r_v(\gamma)\quad\text{for every }v\in F. \end{equation}
In particular,
\begin{equation}\label{eq:intersection-relative-support} G_A=\bigcap_{v\in F}\left(G_A+\mathbb Z^{F\setminus\{v\}}\right). \end{equation}
For $F=\varnothing$ the intersection is understood as all of $\mathbb Z^I$.
\end{lemma}

\begin{proof}
A supported weight has nonnegative frozen coordinates by Lemma~\ref{L:frozen-shifts}. Fix a general presentation using only vertices in $M$,
\[ f:P_A([\gamma]_+)\longrightarrow P_A([-\gamma]_+), \qquad N=\Coker f, \]
and put $H=e_FAe_F$ and $T=Ne_F$. Let $r_v(\gamma)$ be the generic multiplicity of the simple $H$-module at $v$ in $T/T\operatorname{rad}H$. These multiplicities attain their minima on a common dense open subset. Indeed, this top is the cokernel of the map induced by $f$ between the fixed spaces
\[ P_A([\gamma]_+)e_F/ P_A([\gamma]_+)e_F\operatorname{rad}H \longrightarrow P_A([-\gamma]_+)e_F/ P_A([-\gamma]_+)e_F\operatorname{rad}H. \]

For $\eta\ge0$, adding frozen projectives to the source gives an arbitrary $H$-linear map
$\bigoplus_{v\in F}(e_vH)^{\eta(v)}\to T$. 
Every such map lifts to the old target by projectivity, since
\[ \Hom_A(P_v,N)=N(v)=\Hom_H(e_vH,T). \]
The frozen component of the full cokernel is the cokernel of this map. By Nakayama's lemma, a surjection exists exactly when $\eta(v)\ge r_v(\gamma)$ for every $v$. Surjectivity is open, and the projection to the original presentation space is surjective. Hence~\eqref{eq:frozen-fiber-orthant} holds for a general presentation. Saturating every frozen direction except $v$ leaves only the inequality at $v$. Intersecting over all frozen vertices proves~\eqref{eq:intersection-relative-support}.
\end{proof}

\begin{proposition}\label{P:partial-localization-basis}
Let $F\subseteq\Delta_\nu$, and put
\[ G_F(\Delta,W)=G(\Delta,W)+\mathbb Z^F. \]
For $\g\in G_F(\Delta,W)$ choose $\bs a\in\mathbb Z_{\ge0}^{F}$ with $\g+\bs a\in G(\Delta,W)$ and set
\[ C_W^F(\g)=\b x^{-\bs a}C_W(\g+\bs a). \]
This is independent of $\bs a$. If $(\Delta,W)$ is a cluster model, then
\[ \uca_F(\Delta)=\bigoplus_{\g\in G_F(\Delta,W)}k\,C_W^F(\g). \]
Moreover, if $C\subseteq F$, this algebra is free over $R_C=k[x_c^{\pm1}\mid c\in C]$; an $R_C$-basis is obtained by restricting to weights with zero $C$-coordinates.
Taking $F=\Delta_\nu$, every cluster model is a lattice model.
\end{proposition}

\begin{proof}
The choice of shift is immaterial by Lemma~\ref{L:frozen-shifts}. Localizing the ordinary generic basis proves spanning. Multiplying a finite relation by a common sufficiently large monomial in the variables of $F$ gives a relation among distinct ordinary generic characters, proving independence. Each $\mathbb Z^C$-orbit in $G_F$ has a unique representative with zero $C$-coordinates; its characters are the Laurent monomial multiples of that representative over $R_C$. The last assertion follows from Lemma~\ref{L:boundary-intersection} and~\eqref{eq:full-coefficient-saturation}.
\end{proof}

\begin{proposition}[{\cite[Theorem~1.3]{P}}]\label{P:lattice-mutation}
The lattice-model property is invariant under admissible IQP mutation.
\end{proposition}

\begin{proof}
Supported weights and generic characters are compatible with admissible mutation. Tropical mutation commutes with frozen translations, since its correction terms depend only on the mutated coordinate. Thus~\eqref{eq:localized-character} extends the character identity and the bijection on weights to the localized family. The common birational coordinate change preserves the basis property.
\end{proof}

\section{Frozen sink--source extensions}\label{S:frozen-source}

Let $(\Delta,W)$ be an IQP, and let $(\ul\Delta,W)$ be obtained by adding a new frozen sink or source $e$. Thus no cycle of $\ul\Delta$ passes through $e$, and the potential is the old potential $W$. Write
$B(\ul\Delta)=(B(\Delta),b_e)$.
We write $z=x_e$ and use $\ul x_v$ for the variable in the extended seed corresponding to $x_v$.

\subsection{Cyclic homogenization}

\begin{definition}\label{D:rationally-balanced-column}
The new coefficient column $b_e$ is \emph{rationally balanced over $B(\Delta)$} if its class in $\operatorname{coker}B(\Delta)$ is torsion, equivalently if there are $d\ge1$ and $a\in\mathbb Z^{\Delta_0}$ such that
\begin{equation}\label{eq:cyclic-balance} B(\Delta)a^T=-d b_e. \end{equation}
Take $d$ to be the order of $[b_e]$ in the cokernel lattice. Full row rank implies rational balancing.
\end{definition}

Assume throughout this subsection that $b_e$ is rationally balanced, and fix $(a,d)$ as above. Reducing $a$ modulo $d$ gives the old seed a $\mathbb Z/d\mathbb Z$-grading,
\[ \deg_d(x_v)=\overline{a(v)}. \]
Indeed, equation~\eqref{eq:cyclic-balance} implies $B(\Delta)a^T\equiv0\pmod d$, so every exchange relation of $\Delta$ is homogeneous.

If $R=\bigoplus_{\bar r\in\mathbb Z/d\mathbb Z}R_{\bar r}$ is a $\mathbb Z/d\mathbb Z$-graded algebra, define its \emph{cyclic homogenization} by
\[ \operatorname{Cyc}_d(R) :=\bigoplus_{m\in\mathbb Z}R_{\bar m}u^m \subseteq R[u^{\pm1}]. \]
The element $u^d$ is invertible in $\operatorname{Cyc}_d(R)$.

\begin{lemma}\label{L:cyclic-homogenization}
There is an injective homomorphism of ambient fields determined by
\begin{equation}\label{eq:cyclic-map} \Phi(z)=u^d, \qquad \Phi(\ul x_v)=x_vu^{a(v)}\quad(v\in\Delta_0), \end{equation}
which identifies corresponding cluster Laurent rings and induces an isomorphism
\begin{equation}\label{eq:cyclic-UCA} \Phi\bigl(\ucaloc(\ul\Delta)\bigr) =\operatorname{Cyc}_d\bigl(\ucaloc(\Delta)\bigr). \end{equation}
For every $(\g,h)\in\mathbb Z^{\Delta_0}\times\mathbb Z$,
\begin{equation}\label{eq:cyclic-character} \Phi\!\left(C_W^{\mathrm{loc},\ul\Delta}(\g,h)\right) =C_W^{\mathrm{loc},\Delta}(\g)u^{\langle a,\g\rangle+dh}. \end{equation}
\end{lemma}

\begin{proof}
On exponent lattices, $\Phi$ sends $(\mathbf m,q)$ to
$ \bigl(\mathbf m,\langle a,\mathbf m\rangle+dq\bigr)$.
The map on exponent lattices is injective, since the first component is $\mathbf m$ and $d>0$. Thus~\eqref{eq:cyclic-map} defines an injective homomorphism of Laurent fields.

Let $i\in\Delta_\mu$. In the extended exchange relation at $i$, the two exchange monomials acquire under~\eqref{eq:cyclic-map} the respective $u$-exponents
\[ \sum_v[b_{iv}]_+a(v)+d[b_{i,e}]_+, \qquad \sum_v[-b_{iv}]_+a(v)+d[-b_{i,e}]_+. \]
Their difference is $(B(\Delta)a^T+db_e)_i=0$. Hence the image of the exchange relation is the old exchange relation multiplied by a Laurent power of $u$. Induction on mutations at old vertices shows that~\eqref{eq:cyclic-map} identifies each coefficient-localized Laurent ring of $\ul\Delta$ with the cyclic homogenization of the corresponding Laurent ring of $\Delta$.

In a corresponding chart $t$, write $\Phi(\ul x_{v;t})=x_{v;t}u^{a_t(v)}$, with $a_{t_0}=a$. Mutation at $k$ gives
\[ a_{t'}(k)=\sum_{v\in\Delta_0}[b_{kv;t}]_+a_t(v)+d[b_{k,e;t}]_+-a_t(k), \]
while the other entries are unchanged. The balance identity $B(\Delta_t)a_t^T=-d b_{e;t}$ is preserved. Thus, if $\langle a_t,\mathbf m\rangle\equiv r\pmod d$, then
\[ x_t^{\mathbf m}u^r =\Phi\!\left(\ul x_t^{\mathbf m}z^{(r-\langle a_t,\mathbf m\rangle)/d}\right), \]
proving surjectivity onto each homogenized Laurent ring.
The projection onto each $u$-degree preserves every cyclically homogenized cluster Laurent ring. Hence cyclic homogenization commutes with their intersection, proving~\eqref{eq:cyclic-UCA}.

The frozen extension leaves the mutable quotient algebra unchanged, and~\eqref{eq:cyclic-balance} gives $\Phi(\ul y_i)=y_i$. Formula~\eqref{eq:mutable-quotient-character} therefore gives~\eqref{eq:cyclic-character}.
\end{proof}

\subsection{Frozen sink--source extensions}
\begin{theorem}[Frozen sink--source extension]\label{T:frozen-source}
Let $e$ be a new frozen sink or source and assume that $b_e$ is rationally balanced over $B(\Delta)$. Then $(\ul\Delta,W)$ is a lattice model if and only if $(\Delta,W)$ is a lattice model.
\end{theorem}

\begin{proof}
Each old localized character is homogeneous for the cyclic grading, of degree $\overline{\langle a,\g\rangle}$, because the $y$-variables have degree zero. For every integer $m$, the degree-$m$ component of the homogenized algebra is $\ucaloc(\Delta)_{\bar m}u^m$. 
By Lemma~\ref{L:cyclic-homogenization}, the extended characters in this component are exactly
\[ \left\{C_W^{\mathrm{loc},\Delta}(\g)u^m: \langle a,\g\rangle\equiv m\pmod d\right\}. \]
They form a basis in every $u$-degree if and only if the old characters form a basis in every cyclic degree. This proves the equivalence.
\end{proof}

For mutable extensions, we will also need the algebra in which the new source is frozen but its variable is not inverted. Define
\[ \ucarel{\Delta_\nu}(\ul\Delta) =\uca(\ul\Delta)[x_v^{-1}\mid v\in\Delta_\nu]. \]

\begin{proposition}\label{P:relative-source}
Under the hypotheses of Theorem~\ref{T:frozen-source}, suppose that $e$ is a source and that $(\Delta,W)$ is a lattice model. Then $\ucarel{\Delta_\nu}(\ul\Delta)$ has a basis of generic characters parametrized by
$\mathbb Z^{\Delta_0}\times\mathbb Z_{\ge0}$.
More precisely, if
$G_0=\mathbb Z^{\Delta_0}\times\{0\}$,
then the basis is
\[ \left\{x_e^rC_W^{\mathrm{loc}}(\g)\mid \g\in G_0,\ r\ge0\right\}. \]
For every $\g\in G_0$, one may choose the positive shift in the definition of $C_W^{\mathrm{loc}}(\g)$ to be supported entirely on $\Delta_\nu$.
\end{proposition}

\begin{proof}
Theorem~\ref{T:frozen-source} gives the localized basis. By Corollary~\ref{L:rankfree-one-sided}, the characters with $\g(e)\ge0$ form a basis of the relative algebra, and the shifts making their weights supported can be chosen away from $e$. Write each such weight uniquely as $\g+r\e_e$, where $\g\in G_0$ and $r\ge0$, and apply the identity for frozen shifts.
\end{proof}

\section{Saturation and adjoining a cluster variable}\label{S:saturation}

\subsection{Saturation and a one-variable intersection formula}

Let $(\ul\Delta,W)$ be an IQP and let $e\in\ul\Delta_\mu$ be an arbitrary mutable vertex. Write
\[ \Delta=\ul\Delta_{\widehat e},\qquad \ul\Delta^e=\text{the same ice quiver with $e$ frozen}. \]
At this stage, we invert only the coefficients that remain frozen. Put
\[ A_e:=\uca(\ul\Delta^e,\b x) [x_v^{-1}\mid v\in\Delta_\nu]. \]
Write the exchange relation at $e$ as
\begin{equation}\label{eq:general-exchange} x_ex_e'=\rho_e=M_e^++M_e^-, \end{equation}
where $\rho_e$ is independent of $x_e$. Let
\[ K_e=k[x_v^{\pm1}\mid v\in\Delta_\nu] [x_u^{\pm1}\mid u\in\Delta_\mu], \qquad S_e=\left\langle x_u\mid u\in\Delta_\mu\right\rangle. \]
The initial Laurent ring for $\ul\Delta^e$, with the permanent coefficients inverted, is $K_e[x_e]$. Lemma~\ref{L:frozen-prime} gives
\begin{equation}\label{eq:x_e-prime-relative} A_e\cap x_eK_e[x_e]=x_eA_e. \end{equation}
Thus $\bar A_e:=A_e/x_eA_e$ is naturally a subring of $K_e$.

\begin{lemma}\label{L:exchange-adjunction}
The homomorphism
\[ A_e[T]/(x_eT-\rho_e)\longrightarrow A_e[x_e'], \qquad T\longmapsto x_e'=\rho_e/x_e, \]
is an isomorphism.
\end{lemma}

\begin{proof}
The map is surjective. To determine its kernel, take $p(T)=\sum_{j=0}^m a_jT^j$ such that $p(\rho_e/x_e)=0$, with $m\ge1$. Multiplying this equality by $x_e^m$ and reducing modulo $x_eA_e$ gives
$\bar a_m\rho_e^m=0$.
The ring $\bar A_e\subseteq K_e$ is a domain and $\rho_e\ne0$ in it, so $a_m=x_eb$ for some $b\in A_e$. Subtracting $bT^{m-1}(x_eT-\rho_e)$ lowers the degree. Induction, starting with constants, proves that the kernel is generated by $x_eT-\rho_e$.
\end{proof}

\begin{definition}[Saturation at a vertex]\label{D:vertex-saturation}
Let $I_e=(x_e,\rho_e)A_e$. 
We say that $e$ is \emph{saturated over $A_e$} if $I_e$ is $S_e$-saturated, that is,
\[ I_e=A_e\cap S_e^{-1}I_e \]
inside $S_e^{-1}A_e=K_e[x_e]$. Equivalently, the natural map
\[ A_e/(x_e,\rho_e) \longrightarrow S_e^{-1}\!\left(A_e/(x_e,\rho_e)\right) \]
is injective. Under the embedding $\bar A_e\subseteq K_e$, this is equivalent to
\begin{equation}\label{eq:vertex-saturation} \bar A_e\cap\rho_eK_e=\rho_e\bar A_e. \end{equation}
\end{definition}

\begin{proposition}[Saturation criterion]\label{P:rank-one-criterion}
Assume the intersection equality
\begin{equation}\label{eq:rank-one-intersection} \ucaloc(\ul\Delta,\b x) =A_e[x_e^{-1}]\cap K_e[(x_e')^{\pm1}] \end{equation}
inside the common fraction field. Then
\[ \ucaloc(\ul\Delta,\b x)=A_e[x_e'] \quad\Longleftrightarrow\quad e\text{ is saturated over }A_e. \]
The intersection identity holds, in particular, when $B(\ul\Delta)$ has full row rank.
\end{proposition}

\begin{proof}
Assume~\eqref{eq:rank-one-intersection}. We use two elementary consequences of~\eqref{eq:x_e-prime-relative}. First,
\begin{equation}\label{eq:x_e-saturation-relative} A_e=A_e[x_e^{-1}]\cap K_e[x_e]. \end{equation}
Second, in $K_e(x_e)$, with $x_e'=\rho_e/x_e$,
\[ K_e[x_e^{\pm1}]\cap K_e[(x_e')^{\pm1}]=K_e[x_e,x_e']. \]
Indeed, a Laurent polynomial $\sum_i c_ix_e^i$ belongs to $K_e[(x_e')^{\pm1}]$ exactly when $c_{-j}$ is divisible by $\rho_e^j$ for every $j>0$. Hence every element of the intersection has a unique normal form
\begin{equation}\label{eq:normal-form} p(x_e)+\sum_{j=1}^{m}q_j(x_e')^j, \qquad p(x_e)\in K_e[x_e],\quad q_j\in K_e. \end{equation}

Assume first that $e$ is saturated. Induction on $m$ gives
\begin{equation}\label{eq:power-saturation} \bar A_e\cap\rho_e^mK_e=\rho_e^m\bar A_e \qquad(m\ge1). \end{equation}
The inclusion $A_e[x_e']\subseteq\ucaloc(\ul\Delta)$ follows from~\eqref{eq:rank-one-intersection}. Conversely, let $z$ have normal form~\eqref{eq:normal-form}, with $m$ maximal. Since $z\in A_e[x_e^{-1}]$ and $x_e^mz\in K_e[x_e]$, equation~\eqref{eq:x_e-saturation-relative} gives $a:=x_e^mz\in A_e$. Reducing modulo $x_e$ yields
\[ \bar a=q_m\rho_e^m\in\bar A_e. \]
By~\eqref{eq:power-saturation}, $q_m\in\bar A_e$. Choose $b_m\in A_e$ lifting $q_m$. Then $(b_m-q_m)(x_e')^m$ has $x_e'$-degree at most $m-1$, because $b_m-q_m\in x_eK_e[x_e]$ and $x_ex_e'=\rho_e$. Subtracting $b_m(x_e')^m$ and inducting on $m$ proves $z\in A_e[x_e']$.

Conversely, suppose that $\ucaloc(\ul\Delta)=A_e[x_e']$. Let $\bar a=\rho_eq\in\bar A_e\cap\rho_eK_e$, and choose $a\in A_e$ lifting $\bar a$. Since $A_e\subseteq K_e[x_e]$, write
\[ a=\rho_eq+x_er(x_e),\qquad r(x_e)\in K_e[x_e]. \]
Then
$\frac{a}{x_e}=qx_e'+r(x_e)$
belongs to the intersection~\eqref{eq:rank-one-intersection}, hence to $A_e[x_e']$. Write it as $\sum_{j=0}^Nb_j(x_e')^j$ with $b_j\in A_e$. If $N>1$, comparison with the $K_e$-normal form $qx_e'+r(x_e)$ shows that the image of $b_N$ in $\bar A_e$ is zero. Thus $b_N\in A_e\cap x_eK_e[x_e]=x_eA_e$. Writing $b_N=x_ec_N$ replaces $b_N(x_e')^N$ by $\rho_ec_N(x_e')^{N-1}$. Repeating this reduction gives
\[ \frac{a}{x_e}=b_0+b_1x_e', \qquad b_0,b_1\in A_e. \]
Comparison of the coefficient of $x_e'$ in normal form gives $q=\bar b_1\in\bar A_e$. Hence~\eqref{eq:vertex-saturation} holds, and $e$ is saturated.

If $B(\ul\Delta)$ has full row rank, Theorem~\ref{T:upper-bound} applies to both $\ul\Delta$ and $\ul\Delta^e$. Inverting $x_e$ in the finite intersection for the latter gives $A_e[x_e^{-1}]$. The only additional Laurent ring is the one reached by mutation at $e$, namely $K_e[(x_e')^{\pm1}]$. This proves~\eqref{eq:rank-one-intersection}.
\end{proof}

\subsection{Sufficient conditions for saturation}
Let $ \b b_e=(b_{ev})_{v\in\Delta_0}$
be the row of $B(\ul\Delta)$ indexed by $e$, with the $e$-column omitted, and let $\b b_{e,\mu}$ be its restriction to the old mutable columns.

\begin{lemma}\label{L:separating-grading}
The following are equivalent:
\begin{equation}\label{eq:row-separation} \b b_e\notin\operatorname{rowspan}_{\mathbb Q}B(\Delta), \end{equation}
and there exists $\omega\in\mathbb Z^{\Delta_0}$ such that
\[ B(\Delta)\omega^T=0, \qquad \b b_e\cdot\omega\ne0. \]
Under these equivalent conditions, $A_e$, $\bar A_e$, and $K_e$ carry a $\mathbb Z$-grading for which $M_e^+$ and $M_e^-$ have different degrees. Condition~\eqref{eq:row-separation} holds whenever $B(\ul\Delta)$ has full row rank.
\end{lemma}

\begin{proof}
Over $\mathbb Q$, the annihilator of $\ker B(\Delta)$ is the row span of $B(\Delta)$, which proves the equivalence after clearing denominators. Give $x_e$ degree zero and $x_v$ degree $\omega(v)$ for $v\in\Delta_0$. Since
\[ B(\ul\Delta^e)\binom{0}{\omega^T}=B(\Delta)\omega^T=0, \]
every exchange relation of the frozen seed $\ul\Delta^e$ is homogeneous. Thus $A_e$, $\bar A_e$, and $K_e$ are graded, and
\[ \deg M_e^+-\deg M_e^-=\b b_e\cdot\omega\ne0. \]

Finally, suppose that $B(\ul\Delta)$ has full row rank. If $\b b_e=cB(\Delta)$, then restriction to the old mutable columns gives $\b b_{e,\mu}=cB_\mu(\Delta)$. Skew-symmetry implies $c\b b_{e,\mu}^{\,T}=0$, so
\[ c\bigl(-\b b_{e,\mu}^{\,T}\ \ B(\Delta)\bigr)=\bigl(0\ \ \b b_e\bigr), \]
making the $e$-row of $B(\ul\Delta)$ a linear combination of the remaining rows. This contradicts full row rank.
\end{proof}

\begin{proposition}\label{P:one-sided-saturation}
Assume the separating condition~\eqref{eq:row-separation}. If the image in $\bar A_e$ of one of the exchange monomials $M_e^+$ or $M_e^-$ is a unit, then $e$ is saturated over $A_e$. If, in addition, the intersection formula~\eqref{eq:rank-one-intersection} holds, then
\[ \ucaloc(\ul\Delta,\b x)=A_e[x_e']. \]
\end{proposition}

\begin{proof}
Let $U$ be the exchange monomial that is a unit, and let $V$ be the other. Negate the grading if necessary so that $w=U^{-1}V$ is homogeneous of positive degree $d$. Then $\rho_e=U(1+w)$.

Suppose $\bar a=\rho_eq\in\bar A_e$, and write $q=\sum_nq_n\in K_e$ as a finite sum of homogeneous components. Comparing degrees in $U^{-1}\bar a=(1+w)q$ gives
$(U^{-1}\bar a)_n=q_n+wq_{n-d}$.
Induction from the least degree of $q$ shows $q_n\in\bar A_e$ for every $n$. Hence $q\in\bar A_e$, proving saturation. Apply Proposition~\ref{P:rank-one-criterion} for the final assertion.
\end{proof}

\begin{corollary}\label{C:one-sided-saturation}
Assume the separating condition~\eqref{eq:row-separation} and suppose that the row of $e$, restricted to the old mutable columns, is sign-coherent. Then $e$ is saturated over $A_e$. In particular, this applies whenever $e$ is a source or sink in the mutable principal subquiver, even if arrows of the opposite orientation join $e$ to frozen vertices. If $B(\ul\Delta)$ has full row rank, the separating condition is automatic; in that case~\eqref{eq:rank-one-intersection} also holds and $\ucaloc(\ul\Delta)=A_e[x_e']$.
\end{corollary}

\begin{proof}
Sign-coherence means that one of $M_e^+$ and $M_e^-$ involves no mutable variable of $\Delta$. It is therefore a Laurent monomial in the frozen boundary variables, all of which are units in $A_e$ and in $\bar A_e$. Apply Proposition~\ref{P:one-sided-saturation}.
\end{proof}

A second criterion applies when neither exchange monomial is a unit. Put
$R_e:=\bar A_e$ and $K_e=S_e^{-1}R_e$. 
A height-one prime $\mathfrak p$ of $R_e$ is called a \emph{boundary prime} if $\mathfrak p\cap S_e\ne\varnothing$.

\begin{proposition}[Boundary-divisor criterion]\label{P:boundary-divisor-saturation}
Assume that $R_e$ is a Krull domain. Then $e$ is saturated over $A_e$ if and only if
\[ \nu_{\mathfrak p}(\rho_e)=0 \qquad \text{for every height-one boundary prime }\mathfrak p. \]
Equivalently, the principal divisor $\operatorname{div}_{R_e}(\rho_e)$ has no component contained in the old mutable boundary.
\end{proposition}

\begin{proof}
This is the Krull-domain localization criterion. In the intersection description of $R_e$, localization at $S_e$ removes exactly the height-one valuation rings whose primes meet $S_e$. Multiplication by $\rho_e$ can clear a pole introduced by that localization if and only if one of those valuations is positive on $\rho_e$. The converse uses weak approximation for the defining valuations; see \cite[Chapter~2]{Fo}.
\end{proof}

\begin{example}\label{E:saturation-failure}
Consider the mutable quiver
$e\longrightarrow1\longrightarrow2$
and freeze $e$. Deleting $e$ leaves the coefficient-free $A_2$ IQP, hence a lattice model. On the special fiber $x_e=0$, the five finite-type cluster variables specialize to
\[ x_1,\quad x_2,\quad \frac{x_2}{x_1},\quad \frac1{x_1},\quad y:=\frac{1+x_1}{x_2}. \]
The acyclic $A_2$ seed with frozen coefficient $x_e$ has equal cluster and upper cluster algebras \cite{BFZ}, so
\[ R_e\simeq k[x_1^{\pm1},x_2,y]/(x_2y-(1+x_1)). \]
The exchange polynomial at $e$ is $\rho_e=1+x_1=x_2y$. Hence
$\rho_e\frac1{x_2}=y\in R_e$, 
whereas $1/x_2\notin R_e$ because $x_2$ is not a unit. Thus saturation fails. The extended exchange matrix is $3\times3$ skew-symmetric and therefore cannot have full row rank. This shows that the old lattice-model property alone does not force saturation.
\end{example}

\section{Mutable sink--source extensions}\label{S:mutable-source}

\subsection{The source filtration}

Assume now that $e$ is a mutable source. Put
\[ \b a_e=\sum_{e\to u}m_{eu}\e_u\in\mathbb Z^{\Delta_0}, \qquad w_e=\b x^{\b a_e}=\prod_{e\to u}x_u^{m_{eu}}. \]
Then~\eqref{eq:general-exchange} becomes
\[ x_ex_e'=\rho_e, \qquad \rho_e=1+w_e. \]
For brevity write $A=A_e$ and $K=K_e$.

\begin{corollary}\label{C:rank-one}
Assume the separating condition~\eqref{eq:row-separation} and the intersection equality~\eqref{eq:rank-one-intersection}. Then
\[ \ucaloc(\ul\Delta,\b x)=A[x_e']. \]
In particular, this holds when $B(\ul\Delta)$ has full row rank.
\end{corollary}

\begin{proof}
At a source, the exchange monomial $M_e^-=1$ is a unit. Apply Proposition~\ref{P:one-sided-saturation}; the final sentence follows from Lemma~\ref{L:separating-grading} and Proposition~\ref{P:rank-one-criterion}.
\end{proof}

Assume from now on that
\begin{equation}\label{eq:source-polynomiality-assumption} \ucaloc(\ul\Delta)=A[x_e']. \end{equation}
For the statements about generic bases, assume also that $(\Delta,W)$ is a lattice model and that the frozen source column $b_e$ is rationally balanced over $B(\Delta)$. Proposition~\ref{P:relative-source} then gives the relative generic basis of $A$.

Define
$F_h=A+Ax_e'+\cdots+A(x_e')^h\quad(h\ge0)$ and $F_{-1}=0$. 

\begin{lemma}\label{L:source-filtration}
Let $\nu_e$ be the lowest exponent of $x_e$ in the initial Laurent expansion. Then
\[ F_h=\{z\in\ucaloc(\ul\Delta)\mid \nu_e(z)\ge-h\}. \]
For $h\ge1$ and $z\in F_h$,
\[ z\in F_{h-1} \quad\Longleftrightarrow\quad x_e^hz\in x_eA. \]
For $h\ge1$, multiplication by $(x_e')^h$ induces an $A$-module isomorphism
\begin{equation}\label{eq:source-graded-piece} A/x_eA\xrightarrow{\ \sim\ }F_h/F_{h-1}, \qquad \bar a\longmapsto[a(x_e')^h]. \end{equation}
Equivalently, with $A$ in degree zero and $T$ in degree one,
\begin{equation}\label{eq:source-associated-graded} \operatorname{gr}_F\ucaloc(\ul\Delta) \simeq A[T]/(x_eT). \end{equation}
\end{lemma}

\begin{proof}
Since $\nu_e(a)\ge0$ for $a\in A$ and $\nu_e(x_e')=-1$, we have $F_h\subseteq\{z:\nu_e(z)\ge-h\}$. Conversely, write $z\in A[x_e']$ as a finite polynomial in $x_e'$. If its largest degree $m$ is positive and its highest coefficient is divisible by $x_e$, the relation $x_ex_e'=\rho_e$ lowers that degree. Repeat until either $m=0$ or the highest coefficient has valuation zero, using~\eqref{eq:x_e-prime-relative}. In the latter case the term of degree $m$ has valuation $-m$, while all lower-degree terms have valuation at least $-m+1$; hence $\nu_e(z)=-m$. In the former case $z\in A$ and $\nu_e(z)\ge0$. This proves the description of $F_h$.

For $h\ge1$ and $z\in F_h$, the identities $x_e^h(x_e')^j=x_e^{h-j}\rho_e^j$, $0\le j\le h$, imply $x_e^hz\in A$. By~\eqref{eq:x_e-prime-relative},
\[ x_e^hz\in x_eA \quad\Longleftrightarrow\quad \nu_e(z)\ge-h+1 \quad\Longleftrightarrow\quad z\in F_{h-1}. \]
Since $x_e(x_e')^h=\rho_e(x_e')^{h-1}$, the map in~\eqref{eq:source-graded-piece} is well defined; it is surjective by the definition of $F_h$. If $[a(x_e')^h]=0$, then $\nu_e(a)\ge1$, so $a\in x_eA$ by~\eqref{eq:x_e-prime-relative} and $\bar a=0$. These isomorphisms respect multiplication; the degree-zero piece is $A$, proving~\eqref{eq:source-associated-graded}.
\end{proof}

\subsection{A triangular formula for generic characters}

Set
\[ \b q_e=\b a_e-\e_e. \]

Because $e$ is a source, no cycle passes through it, and the potential is supported on the old full subquiver. For right modules there is an exact sequence
\[ 0\longrightarrow Q_e:=\bigoplus_{e\to u}P_u^{m_{eu}} \longrightarrow P_e\longrightarrow S_e\longrightarrow0. \]
Thus $\b q_e$ is the $g$-vector of the simple $S_e$.

\begin{lemma}[{\cite{DF}}]\label{L:generic-cancellation}
Let $U_1\to U_0$ be any projective presentation pair of weight $\delta$. A general presentation in $\Hom_J(U_1,U_0)$ is homotopy equivalent to a general presentation in the reduced space $\PHom_J(\delta)$. In particular, the generic cokernel depends only on $\delta$.
\end{lemma}

Let $G_0=\mathbb Z^{\Delta_0}\times\{0\}$. 
By Proposition~\ref{P:relative-source}, the generic-character basis of $A$ consists of
\[ x_e^r C_W^{\mathrm{loc},\ul\Delta^e}(\g), \qquad \g\in G_0,\quad r\ge0. \]

\begin{proposition}[Triangularity at a source]\label{P:source-triangular}
For every $\g\in G_0$ and $h\ge0$,
\begin{equation}\label{eq:source-triangular} C_W^{\mathrm{loc},\ul\Delta^e}(\g)(x_e')^h \equiv C_W^{\mathrm{loc},\ul\Delta}(\g+h\b q_e) \pmod{F_{h-1}}. \end{equation}
\end{proposition}

\begin{proof}

For $h=0$, the two characters coincide. Let $h\ge1$. By Proposition~\ref{P:relative-source}, a shift at the old frozen vertices makes $\g$ supported for $\ul\Delta^e$. The original congruence follows by multiplying by the inverse coefficient monomial. We may therefore assume $\g\in G(\ul\Delta^e,W)$.

Choose a general old presentation
$ f:P_1\longrightarrow P_0$
of weight $\g$, with generic cokernel $M$. Since $\g(e)=0$, neither $P_1$ nor $P_0$ has a $P_e$-summand. For right modules, every nontrivial path beginning at the source $e$ has a unique first arrow $e\to u$, and no Jacobian relation involves an arrow incident with $e$. Hence
\[ P_e|_{\Delta}\simeq Q_e. \]
The nonreduced presentation space
$P_1\oplus Q_e^h \longrightarrow P_0\oplus P_e^h$
has weight $\g+h\b q_e$. By Lemma~\ref{L:generic-cancellation}, its general cokernel is a general cokernel of that weight. After restriction to $\Delta$, the block $Q_e^h\to Q_e^h$ is generically invertible, and the map taking a presentation to its Schur complement is dominant onto the old presentation space. Intersect this open set with the generic loci for the finitely many quotient Grassmannians. The resulting cokernel $M_h$ satisfies
\[ M_h|_{\Delta}\simeq M, \qquad M_h(e)\simeq k^h. \]
The module $M_h$ still vanishes at the old frozen vertices, so it is supported for the extended seed and gives its generic character.

For a quotient of $M_h$, write $(\d,r)$ for its dimension vector on the old vertices and at $e$, and put
\[ \lambda_e(\d)=\sum_{e\to u}m_{eu}\d(u). \]
The corresponding monomial in the generic character has $x_e$-valuation
$ -h+\lambda_e(\d)\ge-h.$
Thus the character lies in $F_h$ by Lemma~\ref{L:source-filtration}. A term has exponent $-h$ exactly when $\lambda_e(\d)=0$. In this case the quotient vanishes at every vertex joined to $e$ by an outgoing arrow, so the arrow conditions impose no restriction on the kernel at $e$. Consequently
\[ \Gr^{(\d,r)}(M_h) \simeq \Gr^\d(M)\times\Gr^r(k^h) \qquad\text{if }\lambda_e(\d)=0. \]

Now $y_e=w_e^{-1}$, and the monomial factor in the character contributes $w_e^hx_e^{-h}$. Since $x_e^hF_h\subseteq A$, multiplying both expressions by $x_e^h$ gives elements of $A$. Taking the terms of $x_e$-exponent zero and using $\chi(\Gr^r(k^h))=\binom hr$ gives
\begin{align*}
 x_e^h C_W^{\mathrm{loc},\ul\Delta}(\g+h\b q_e)
 &\equiv
 \b x^\g w_e^h
 \sum_{\lambda_e(\d)=0}
 \chi(\Gr^\d(M))\b y^\d
 \sum_{r=0}^h\binom hrw_e^{-r}\\
 &=
 \rho_e^h\b x^\g
 \sum_{\lambda_e(\d)=0}
 \chi(\Gr^\d(M))\b y^\d\\
 &\equiv
 \rho_e^hC_W^{\mathrm{loc},\ul\Delta^e}(\g)
 \pmod{x_eK[x_e]}.
\end{align*}
Their difference belongs to $A\cap x_eK[x_e]=x_eA$. Since
\[ x_e^h C_W^{\mathrm{loc},\ul\Delta^e}(\g)(x_e')^h =\rho_e^hC_W^{\mathrm{loc},\ul\Delta^e}(\g), \]
Lemma~\ref{L:source-filtration} proves~\eqref{eq:source-triangular}.

\end{proof}

\subsection{Mutable sink--source extension}

\begin{theorem}\label{T:mutable-source}
Assume that $(\Delta,W)$ is a lattice model, that the source column $b_e$ is rationally balanced over $B(\Delta)$, and that~\eqref{eq:source-polynomiality-assumption} holds. Then $(\ul\Delta,W)$ is a lattice model.
\end{theorem}

\begin{proof}
Put $b_\g=C_W^{\mathrm{loc},\ul\Delta^e}(\g)$ for $\g\in G_0$. Proposition~\ref{P:relative-source} gives the basis
$\{x_e^r b_\g:\g\in G_0,\ r\ge0\}$
of $F_0=A$. At the source these are precisely the full characters
$C_W^{\mathrm{loc},\ul\Delta}(\g+r\e_e)$:
the additional source projectives are positive decorations and leave the old cokernel unchanged.

The classes of $b_\g$, $\g\in G_0$, form a basis of $A/x_eA$. By~\eqref{eq:source-graded-piece}, their products with $(x_e')^h$ form a basis of $F_h/F_{h-1}$ for every $h\ge1$. Proposition~\ref{P:source-triangular} identifies these classes with those of
$C_W^{\mathrm{loc},\ul\Delta}(\g+h\b q_e)$.
These characters therefore give bases of the successive quotients of the filtration. Since the filtration is exhaustive and indexed by nonnegative integers, induction on the degree proves spanning. Taking the largest degree in a finite relation proves independence.

The weights of the resulting basis are
\[ \{\g+r\e_e:\g\in G_0,\ r\ge0\} \ \sqcup\ \{\g+h\b q_e:\g\in G_0,\ h\ge1\} =\mathbb Z^{\ul\Delta_0}. \]
The two sets are distinguished by the sign of the $e$-coordinate, and their old coordinates are arbitrary. Thus the basis consists of all localized generic characters.
\end{proof}

\begin{corollary}[Mutable sink--source extension]\label{C:mutable-source-fullrank}\label{C:mutable-sink}
Let $e$ be a mutable sink or source and put $\Delta=\ul\Delta_{\widehat e}$. If $e$ is a sink, mutate once at $e$. Suppose the resulting source extension satisfies the hypotheses of Theorem~\ref{T:mutable-source}. Then $(\ul\Delta,W)$ is a lattice model. In either orientation, it is enough that $B(\Delta)$ and $B(\ul\Delta)$ have full row rank and that $(\Delta,W)$ is a lattice model.
\end{corollary}

\begin{proof}
If $e$ is a source, apply Theorem~\ref{T:mutable-source}. If $e$ is a sink, mutation at $e$ creates no composite arrows, leaves the old full sub-IQP unchanged, and makes $e$ a source. Apply the theorem and mutate back. Under the rank assumptions, $b_e$ is rationally balanced over $B(\Delta)$, while the separating grading and intersection formula give the required algebra equality by Corollary~\ref{C:rank-one}. These conditions also hold after sink mutation, since the old sub-IQP is unchanged and matrix mutation preserves rank.
\end{proof}

\section{Iterated extensions}\label{S:iterated-localized}

\begin{lemma}[Rigidity along a sink--source extension]\label{L:rigid-extension}
Let $(\Delta,W)$ be a reduced rigid Jacobi-finite QP, and let $(\ul\Delta,W)$ be obtained by adjoining one vertex $e$ as a source or a sink, without changing the potential. Then $(\ul\Delta,W)$ is again reduced, rigid, and Jacobi-finite. Consequently, every construction obtained from a rigid Jacobi-finite base by sink--source extensions and reduced QP mutations is admissible; all intermediate reduced QPs are rigid and Jacobi-finite.
\end{lemma}

\begin{proof}
Every cycle is supported on $\Delta$, so rigidity is preserved, and no new $2$-cycle is created. The old Jacobian relations are unchanged. At a source, each path through $e$ starts with an arrow $e\to u$ followed by a path in $J(\Delta,W)$; the new path spaces are therefore finite-dimensional. The sink case is dual. Reduced mutation preserves rigidity and Jacobi-finiteness \cite[Corollaries~6.6 and~6.11]{DWZ1}, and rigid QPs are nondegenerate \cite[Proposition~8.1]{DWZ1}. Induction proves the final assertion.
\end{proof}

Restriction to a full sub-QP preserves rigidity and nondegeneracy \cite[Definition~8.8 and Proposition~8.9]{DWZ1}.

\begin{lemma}[Mutation and restriction]\label{L:mutation-restriction}
Let $(\Delta,W)$ be the full sub-IQP of $(\ul\Delta,\ul W)$ on a vertex set $I$, and let $k\in I$ be mutable. Then
\[ \mu_k(\ul\Delta,\ul W)|_I \simeq_{\mathrm{r.e.}} \mu_k(\Delta,W). \]
Consequently, the same statement holds for every admissible sequence of mutations supported on $I$.
\end{lemma}

\begin{proof}
Restriction to a full vertex subset commutes with DWZ premutation \cite[Section~5]{DWZ1}. For reduction, apply the splitting theorem \cite[Theorem~4.6]{DWZ1}. The linear part of a right-equivalence and the quadratic trivial summand decompose according to the endpoints of arrows. Restriction keeps each block whose endpoints lie in $I$. The restricted linear part is therefore invertible by \cite[Proposition~2.4]{DWZ1}; the trivial summand remains trivial, and the reduced summand remains reduced. Uniqueness of the reduced part up to right-equivalence proves the assertion for one mutation. Iteration proves it for a sequence supported on $I$.
\end{proof}

\begin{definition}\label{D:old-vertex-optimization}
Let $(\Delta,W)$ be a full sub-IQP of $(\ul\Delta,\ul W)$, and let $e$ be a vertex of $\ul\Delta\setminus\Delta$. An admissible sequence \emph{optimizes $e$ over $(\Delta,W)$} if it mutates only vertices of $\Delta_\mu$ and makes $e$ a source or a sink.
\end{definition}

\begin{theorem}[Iterated extension]\label{T:iterated-localized}
Let $\bs e=\{e_1,\ldots,e_N\}$ be the new vertices, let $\Delta_i$ be the partial extension on $\Delta_0\sqcup\{e_1,\ldots,e_i\}$, and let $(\Delta_i^\star,W_i^\star)$ be an optimized seed at the $i$-th step; when $e_i$ is mutable, orient it as a source by one additional mutation at $e_i$ if necessary. Assume that:
\begin{enumerate}
\item every optimizing sequence mutates only previously present mutable vertices;
\item if $e_i$ remains frozen, its coefficient column in $B(\Delta_i^\star)$ is rationally balanced over the deleted seed;
\item if $e_i$ is mutable, then after orienting it as a source if necessary, its temporarily frozen coefficient column is rationally balanced and
\[ \ucaloc(\Delta_i^\star)=A_i[x_{e_i}'] \]
for the corresponding relative frozen algebra $A_i$.
\end{enumerate}
If $(\Delta,W)$ is a lattice model, then $(\ul\Delta,\ul W)$ is a lattice model.
\end{theorem}

\begin{proof}
Induct on $i$. By Lemma~\ref{L:mutation-restriction}, restriction of the optimized seed to the old vertices gives a mutation of $(\Delta_{i-1},W_{i-1})$, which is a lattice model. Apply Theorem~\ref{T:frozen-source} if the new vertex is frozen and Theorem~\ref{T:mutable-source} if it is a mutable source. Corollary~\ref{C:mutable-sink} treats a mutable sink. Reverse the optimizing sequence and use mutation invariance.
\end{proof}

\begin{corollary}[Full-rank iterated extension]\label{C:iterated-fullrank}
Assume that every partial exchange matrix $B(\Delta_i)$ has full row rank, every new frozen vertex can be optimized to a source or sink, and every new mutable vertex can be optimized to a source or sink, always using only previously present mutable vertices. If $(\Delta,W)$ is a lattice model, then $(\ul\Delta,\ul W)$ is a lattice model.
\end{corollary}

\begin{proof}
Full row rank of the matrix after deleting the new vertex implies rational balancing of its coefficient column. If the new vertex is mutable, orient it as a source. Full row rank of the extended matrix gives the intersection formula and a separating grading, so Corollary~\ref{C:rank-one} gives the required algebra equality. Apply Theorem~\ref{T:iterated-localized}.
\end{proof}

\section{Frozen amalgamation}\label{S:amalgamation}

\begin{definition}[Frozen amalgamation]\label{D:frozen-amalgamation}
Let $(\Delta_i,W_i)$, $i=1,2$, be IQPs with a common set $C$ of frozen vertices, and assume that their vertex sets are otherwise disjoint. A \emph{frozen amalgamation} is an IQP $(\Delta,W)$ such that
\begin{enumerate}
\item $\Delta_0=(\Delta_1)_0\cup_C(\Delta_2)_0$ and $\Delta_\mu=(\Delta_1)_\mu\sqcup(\Delta_2)_\mu$;
\item after suppressing arrows with both endpoints in $C$, the full subquiver of $\Delta$ on $(\Delta_i)_0$ is $\Delta_i$, and no arrow joins a vertex of $(\Delta_1)_0\setminus C$ to a vertex of $(\Delta_2)_0\setminus C$;
\item writing $J=J(\Delta,W)$, $J_i=J(\Delta_i,W_i)$, and $e_C=\sum_{c\in C}e_c$, the natural restriction maps induce an isomorphism
\begin{equation}\label{eq:amalgamated-Jacobian-quotient} J/Je_CJ \simeq (J_1/J_1e_CJ_1)\times(J_2/J_2e_CJ_2). \end{equation}
\end{enumerate}
Arrows with both endpoints in $C$, and cyclic terms meeting $C$, are allowed.  They do not affect~\eqref{eq:amalgamated-Jacobian-quotient}.
\end{definition}

Put
\[ R_C=k[x_c^{\pm1}\mid c\in C]. \]
The coefficient-localized upper cluster algebras of both factors are naturally $R_C$-algebras.  Since mutations in the two factors involve disjoint mutable sets, every seed of the amalgamation is a pair of factor seeds, and its Laurent ring is their tensor product over $R_C$.

\begin{lemma}\label{L:free-intersection}
Let $R$ be a domain, let $(M_i)_{i\in I}$ be submodules of a common $R$-module $M$, and let $N$ be a free $R$-module. Then, inside $M\otimes_RN$,
\[ \left(\bigcap_{i\in I}M_i\right)\otimes_RN = \bigcap_{i\in I}(M_i\otimes_RN). \]
\end{lemma}

\begin{proof}
Choose a basis $J$ of $N$. Under $M\otimes_RN\simeq\bigoplus_{j\in J}M$, the submodule $M_i\otimes_RN$ is $\bigoplus_{j\in J}M_i$. Their intersection is therefore $\bigoplus_{j\in J}(\bigcap_iM_i)$, which is the left side.
\end{proof}

\begin{theorem}[Frozen amalgamation of lattice models]\label{T:frozen-amalgamation}
Let $(\Delta,W)$ be a frozen amalgamation of $(\Delta_1,W_1)$ and $(\Delta_2,W_2)$ along $C$. Assume that each $(\Delta_i,W_i)$ is a lattice model. Then
\begin{equation}\label{eq:amalgamated-UCA-tensor} \ucaloc(\Delta,W) \simeq \ucaloc(\Delta_1,W_1) \otimes_{R_C} \ucaloc(\Delta_2,W_2). \end{equation}
Moreover, $(\Delta,W)$ is a lattice model.  If $\g_i\in\mathbb Z^{(\Delta_i)_0}$ and their common coordinates are added under the identification of $C$, then
\begin{equation}\label{eq:amalgamated-character-product} C_W^{\mathrm{loc}}(\g_1\#\g_2) = C_{W_1}^{\mathrm{loc}}(\g_1) C_{W_2}^{\mathrm{loc}}(\g_2). \end{equation}
Here $\g_1\#\g_2$ has the noncommon coordinates of $\g_i$ and, at $c\in C$, coordinate $\g_1(c)+\g_2(c)$.
\end{theorem}

\begin{proof}
First consider the upper cluster algebras. By frozen shifts, the lattice basis of the $i$-th factor is free over $R_C$: an $R_C$-basis is
\[ \mathcal B_i^0=\{C_{W_i}^{\mathrm{loc}}(\g_i)\mid \g_i(c)=0\text{ for every }c\in C\}. \]
The Laurent rings of both factors are free over $R_C$. Let $(L_{i,\sigma_i})_{\sigma_i}$ be the cluster Laurent rings of the $i$-th factor. Every seed of the amalgamation is a pair of factor seeds, so
\[ \ucaloc(\Delta)=\bigcap_{\sigma_1,\sigma_2} \bigl(L_{1,\sigma_1}\otimes_{R_C}L_{2,\sigma_2}\bigr). \]
Apply Lemma~\ref{L:free-intersection} first with the free module $L_{2,\sigma_2}$ and then with the free module $\ucaloc(\Delta_1)$:
\begin{align*}
 \ucaloc(\Delta)
 &=\bigcap_{\sigma_2}\left(
    \left(\bigcap_{\sigma_1}L_{1,\sigma_1}\right)
    \otimes_{R_C}L_{2,\sigma_2}\right)\\
 &=\ucaloc(\Delta_1)\otimes_{R_C}
   \left(\bigcap_{\sigma_2}L_{2,\sigma_2}\right)\\
 &=\ucaloc(\Delta_1)\otimes_{R_C}\ucaloc(\Delta_2).
\end{align*}
This proves~\eqref{eq:amalgamated-UCA-tensor}.

Quotienting~\eqref{eq:amalgamated-Jacobian-quotient} by the remaining frozen idempotents gives
\[ J_\mu\simeq (J_1)_\mu\times(J_2)_\mu. \]
A general presentation over this product is a pair of general presentations, one for each factor. Cokernels and quotients split accordingly, so the quotient $F$-polynomials multiply. The mutable $y$-variables are unchanged. Multiplying the monomial factors adds their exponents at the common frozen vertices. Formula~\eqref{eq:mutable-quotient-character} gives~\eqref{eq:amalgamated-character-product}.

Products of the two $R_C$-bases $\mathcal B_i^0$ form an $R_C$-basis of the tensor product in~\eqref{eq:amalgamated-UCA-tensor}.  By~\eqref{eq:amalgamated-character-product}, these products are the localized generic characters with zero common coordinate; Laurent monomials in the common frozen variables give the remaining weights. Hence $(\Delta,W)$ is a lattice model.
\end{proof}

\begin{corollary}[Relative version]\label{C:relative-amalgamation}
Under the hypotheses of Theorem~\ref{T:frozen-amalgamation}, assume also that both factors are ordinary cluster models. For $i=1,2$, let
\[ C\subseteq F_i\subseteq(\Delta_i)_\nu, \qquad F=F_1\cup_C F_2. \]
Then
\begin{equation}\label{eq:relative-amalgamation} \uca_F(\Delta,W) \simeq \uca_{F_1}(\Delta_1,W_1) \otimes_{R_C} \uca_{F_2}(\Delta_2,W_2). \end{equation}
If
\[ G_{F_i}^0= \{\g_i\in G_{F_i}(\Delta_i,W_i) \mid \g_i(c)=0\text{ for all }c\in C\}, \]
then an $R_C$-basis of~\eqref{eq:relative-amalgamation} is
\[ \left\{ C_{W_1}^{F_1}(\g_1) C_{W_2}^{F_2}(\g_2) \ \middle|\ \g_i\in G_{F_i}^0 \right\}. \]
After multiplying by Laurent monomials in the common variables, these products are the partially localized generic characters of the amalgamation.
\end{corollary}

\begin{proof}
By Proposition~\ref{P:partial-localization-basis}, the relative algebras of both factors are free over $R_C$. Lemma~\ref{L:boundary-intersection} expresses them as intersections of the corresponding Laurent rings. Apply Lemma~\ref{L:free-intersection} twice. Apply Lemma~\ref{L:idempotent-quotient-shifts} first to the common frozen vertices and then to the other inverted coefficients; the partially supported weights then separate into the two factor sets. The character product formula gives the stated basis.
\end{proof}

The theorem may be iterated when the common frozen sets are amalgamated successively and~\eqref{eq:amalgamated-Jacobian-quotient} holds at each stage; compare \cite{FG}.

\section{Boundary valuations and polyhedral models}\label{S:boundary}

Let $(\Delta,W)$ be a lattice model, with localized generic basis
\[ \mathcal B^{\mathrm{loc}} =\{C_W^{\mathrm{loc}}(\g)\mid\g\in\mathbb Z^{\Delta_0}\}. \]

\begin{definition}\label{D:optimized-boundary}
A frozen vertex $v$ is \emph{optimizable} if an admissible mutation sequence $\bs\mu_v$, avoiding $v$, makes it a source or sink after the other frozen vertices are deleted. Set $\g_v=\bs\mu_v^{\mathrm{trop}}(\g)$ and define, for integral weights,
\[
\vartheta_v(\g)=
\begin{cases}
\g_v(v),&v\text{ is a source in the optimized seed},\\[1mm]
\g_v^\vee(v),&v\text{ is a sink in the optimized seed}.
\end{cases}
\]
\end{definition}

\begin{proposition}\label{L:optimized-valuation}
For an optimizable frozen vertex $v$,
\begin{equation}\label{eq:optimized-boundary-value} \nu_v\bigl(C_W^{\mathrm{loc}}(\g)\bigr)=\vartheta_v(\g). \end{equation}
The localized generic basis is valuatively independent for $\nu_v$, and
\[ \vartheta_v(\g)\ge0\quad\Longleftrightarrow\quad\g\in G_{\widehat v}. \]
In particular, the boundary function on integral weights is independent of the optimizing sequence.
\end{proposition}

\begin{proof}
By Lemma~\ref{L:boundary-intersection}, mutations away from $v$ leave both $x_v$ and its valuation unchanged. Mutation invariance of the generic basis and Corollary~\ref{L:rankfree-one-sided} give the valuation formula and valuative independence. As in the proof of Proposition~\ref{P:lattice-mutation}, tropical mutation commutes with frozen translations. It therefore sends
$G+\mathbb Z^{\Delta_\nu\setminus\{v\}}$
to the corresponding partially saturated set in the new seed. Apply~\eqref{eq:one-sided-relative-support}. Since the left side of~\eqref{eq:optimized-boundary-value} does not depend on the sequence, neither does the boundary function.
\end{proof}

\subsection{The ordinary upper cluster algebra}

\begin{theorem}[Boundary compactification]\label{T:boundary-compactification}
Assume that every frozen vertex is optimizable. Then
\[ G(\Delta,W) =\{\g\in\mathbb Z^{\Delta_0}\mid \vartheta_v(\g)\ge0\text{ for every }v\in\Delta_\nu\}, \]
and
\[ \uca(\Delta,\b x) =\bigoplus_{\g\in G(\Delta,W)}k\,C_W(\g) =\bigoplus_{\substack{\g\in\mathbb Z^{\Delta_0}\\ \vartheta_v(\g)\ge0\ \forall v\in\Delta_\nu}} k\,C_W^{\mathrm{loc}}(\g). \]
Thus $(\Delta,W)$ is an ordinary cluster model.
\end{theorem}

\begin{proof}
Proposition~\ref{L:optimized-valuation} and Lemma~\ref{L:frozen-corner-saturation} give
\[ \{\g:\vartheta_v(\g)\ge0\ \forall v\} =\bigcap_{v\in\Delta_\nu}G_{\widehat v} =G(\Delta,W). \]
By valuative independence, an element of the localized algebra has nonnegative frozen valuations precisely when every basis element in its expansion satisfies these inequalities. Apply Lemma~\ref{L:boundary-intersection}; on $G(\Delta,W)$ the localized and ordinary characters coincide.
\end{proof}

\subsection{Polyhedrality}

\begin{lemma}[{\cite[Proposition~5.15]{Fs1}}]\label{L:supported-additivity}
The set $G(\Delta,W)$ is an additive submonoid of $\mathbb Z^{\Delta_0}$, and for every admissible mutation at a mutable vertex $u$,
\[ G\bigl(\mu_u(\Delta,W)\bigr) =\mu_u^{\mathrm{trop}}\bigl(G(\Delta,W)\bigr). \]
In particular, each $G_{\widehat v}=G(\Delta,W)+\mathbb Z^{\Delta_\nu\setminus\{v\}}$ is additive.
\end{lemma}

\begin{lemma}\label{L:additive-polyhedral-union}
Let $C=C_1\cup\cdots\cup C_s\subseteq\mathbb R^n$ be a nonempty finite union of closed rational polyhedral cones. If $C\cap\mathbb Z^n$ is closed under addition, then $C$ is a rational convex polyhedral cone.
\end{lemma}

\begin{proof}
For rational $x,y\in C$, choose $N\ge1$ with $Nx,Ny\in\mathbb Z^n$. Additivity gives $N(x+y)\in C$, and positive scaling gives $x+y\in C$. Rational points are dense in each $C_i$, and $C$ is closed because the union is finite. Thus $C+C\subseteq C$. Together with closure under nonnegative scaling, this proves convexity. Finally,
\[ C=C_1+\cdots+C_s. \]
The union of finite rational generating sets for the $C_i$ generates this sum, so it is rational polyhedral.
\end{proof}

\begin{corollary}[Piecewise-linear boundary criterion]\label{C:polyhedral-boundary}
Under the hypotheses of Theorem~\ref{T:boundary-compactification}, suppose that, for every frozen $v$, there is a continuous positively homogeneous function
$q_v:\mathbb R^{\Delta_0}\to\mathbb R$, linear with rational coefficients on each cone of a finite complete rational polyhedral fan, such that
\begin{equation}\label{eq:piecewise-boundary-description} q_v(\g)\ge0\quad\Longleftrightarrow\quad \vartheta_v(\g)\ge0 \qquad(\g\in\mathbb Z^{\Delta_0}). \end{equation}
Then $(\Delta,W)$ is a polyhedral cluster model, with
\[ \mathcal C_\Delta =\{\xi\in\mathbb R^{\Delta_0}:q_v(\xi)\ge0 \text{ for every }v\in\Delta_\nu\}, \qquad G(\Delta,W)=\mathcal C_\Delta\cap\mathbb Z^{\Delta_0}. \]
\end{corollary}

\begin{proof}
Take a common finite refinement of the linearity fans. On each cone, every condition $q_v\ge0$ is a rational homogeneous linear inequality. Hence $\mathcal C_\Delta$ is a finite union of closed rational polyhedral cones. Its integral points are $G(\Delta,W)$ by Theorem~\ref{T:boundary-compactification}, and this set is additive by Lemma~\ref{L:supported-additivity}. Lemma~\ref{L:additive-polyhedral-union} proves convex polyhedrality. The basis statement follows from Theorem~\ref{T:boundary-compactification}.
\end{proof}

\begin{definition}\label{D:relative-source-optimization}
A frozen vertex $v$ \emph{admits a relative source optimization} if a finite admissible mutation sequence $\bs\sigma_v$ at mutable vertices makes $v$ a source in the full subquiver on $\Delta_\mu\cup\{v\}$. Thus arrows joining $v$ to other frozen vertices are irrelevant. The sequence may depend on $v$.
\end{definition}

\begin{proposition}\label{P:source-polyhedral-support}
Let $(\Delta,W)$ be a reduced Jacobi-finite nondegenerate IQP. Suppose that every frozen vertex admits a relative source optimization, and choose sequences $\bs\sigma_v$ as in Definition~\ref{D:relative-source-optimization}. Put
\[ q_v(\xi)=\bigl(\bs\sigma_v^{\mathrm{trop}}(\xi)\bigr)(v), \qquad C=\{\xi\in\mathbb R^{\Delta_0}:q_v(\xi)\ge0 \text{ for every frozen }v\}. \]
Then $C$ is a rational convex polyhedral cone and
\[ G(\Delta,W)=C\cap\mathbb Z^{\Delta_0}. \]
Under the hypotheses of Theorem~\ref{T:boundary-compactification}, this gives a polyhedral generic basis.
\end{proposition}

\begin{proof}
Put $F=\Delta_\nu$ and $J_v=J/Je_{F\setminus\{v\}}J$. By Lemma~\ref{L:idempotent-quotient-shifts}, $\g\in G_{\widehat v}$ exactly when a general $J_v$-presentation of weight $\g|_{\Delta_\mu\cup\{v\}}$ has cokernel zero at $v$. Restriction to this full sub-IQP commutes with the chosen mutations by Lemma~\ref{L:mutation-restriction}. Supported weights are mutation invariant, and tropical mutation commutes with frozen translations; see \cite[Remark~3.1 and Lemma~5.12]{Fc1}, with $\delta=-\g$. In the resulting quiver, $v$ is the only frozen vertex and is a source. No other projective has a component at $v$, so support is equivalent to nonnegativity of the $v$-coordinate. Hence
\[ G_{\widehat v}=\{\g\in\mathbb Z^{\Delta_0}:q_v(\g)\ge0\}. \]
Lemma~\ref{L:frozen-corner-saturation} gives
$G(\Delta,W)=\bigcap_{v\in F}G_{\widehat v}=C\cap\mathbb Z^{\Delta_0}$.

A finite sequence of tropical mutations is continuous, positively homogeneous, and integral linear on each cone of a finite complete rational polyhedral fan. Thus $C$ is a finite union of closed rational polyhedral cones. Lemmas~\ref{L:supported-additivity} and~\ref{L:additive-polyhedral-union} show that it is convex and polyhedral. Under the hypotheses of Theorem~\ref{T:boundary-compactification}, the support equality above gives~\eqref{eq:piecewise-boundary-description}, so Corollary~\ref{C:polyhedral-boundary} applies. If $F=\varnothing$, the intersection is all of $\mathbb R^{\Delta_0}$.
\end{proof}

\begin{corollary}\label{C:extension-compactification}
Assume the hypotheses of Theorem~\ref{T:iterated-localized}. If every frozen vertex of the final IQP is optimizable to a source or sink, the final IQP is a cluster model. If, in addition, every frozen vertex admits a relative source optimization, it is a polyhedral cluster model. More generally, the latter conclusion holds under the piecewise-linear boundary criterion of Corollary~\ref{C:polyhedral-boundary}.
\end{corollary}

\begin{proof}
Apply Theorem~\ref{T:iterated-localized} and then Theorem~\ref{T:boundary-compactification}. For polyhedrality use Proposition~\ref{P:source-polyhedral-support}, or Corollary~\ref{C:polyhedral-boundary} in the more general case.
\end{proof}

\section{Applications to hive quivers and flag semi-invariants}\label{S:applications}

\subsection{Hive quivers}\label{S:glue}

Fix $l\ge2$. For an oriented triangle $\mathsf t$, let $\Delta_l^{\mathsf t}$ be the standard size-$l$ ice hive quiver of \cite{Fs1}; the terminology comes from the Knutson--Tao hive model \cite{KT}. Its vertices are triples $(i,j,k)\in\mathbb Z_{\ge0}^3$ with $i+j+k=l$ and at most one zero coordinate; the vertices with a zero coordinate are frozen. Figure~\ref{f:hive5} shows the case $l=5$.
\begin{figure}[t]
\centering
$\hivefive$
\caption{The size-$5$ hive quiver; the three boundary edges are dashed.}
\label{f:hive5}
\end{figure}

For a consistent triangulation $\mb T$ of the disk $\mc D_m$, form $\Diamond_l(\mb T)$ from the full triangular grids by identifying internal edges, unfreezing their vertices, cancelling opposite arrows, and discarding arrows between vertices that remain frozen. All grid arrows incident with newly mutable vertices are retained. This specifies the exchange matrix; frozen--frozen arrows in the IQP constructed below are retained under mutation. Write $\Diamond_l^{\mathrm{fr}}(\mb T)$ when the internal edges remain frozen. Appendix~\ref{A:hive-combinatorics} gives the local operations and the rank calculation.

The standard hive IQP $(\Delta_l,W_l)$ has full row rank, is rigid and Jacobi-finite, and is a polyhedral cluster model \cite[Lemmas~6.1 and~6.4, Theorem~6.8, and Corollary~8.12]{Fs1}.

\begin{proposition}\label{P:frozen-hive-amalgamation}
Equip $\Diamond_l^{\mathrm{fr}}(\mb T)$ with a rigid Jacobi-finite potential whose quotient by the idempotents on the internal edges is the product of the one-triangle hive quotients. Then
$ (\Diamond_l^{\mathrm{fr}}(\mb T),W_l^{\mathrm{fr}}(\mb T))$
is a lattice model.
\end{proposition}

\begin{proof}
The mutable interiors of distinct triangles are disjoint, and every path or cyclic term meeting two triangles passes through an internal-edge vertex. Thus the gluing is a frozen amalgamation. Each one-triangle hive IQP is a polyhedral cluster model by \cite{Fs1}, hence a lattice model by Proposition~\ref{P:partial-localization-basis}. Iterate Theorem~\ref{T:frozen-amalgamation} along the dual tree of the triangulation.
\end{proof}

Let $\Delta_l^{\flat\flat}$ be the hive obtained by deleting two boundary edges, as in \cite[Example~3.12]{Fs2}, with restricted potential $W_l^{\flat\flat}=W_l|_{\Delta_l^{\flat\flat}}$. It is an ordinary polyhedral cluster model by \cite[Theorem~5.6, Remark~5.7, Corollary~5.8, and Example~5.12]{Fs2}, hence a lattice model by Proposition~\ref{P:partial-localization-basis}. Restriction preserves rigidity, and the restricted Jacobian algebra is a finite-dimensional idempotent quotient. Full row rank follows from \cite[Lemma~2.10]{Fs2} and \cite[Lemma~6.1]{Fs1}. We use the quivers in these references without adding arrows between frozen vertices.

\begin{proposition}\label{P:hive-sequence}
For every consistent triangulation $\mb T$, there is a rigid Jacobi-finite IQP $(Q_l(\mb T),W_l(\mb T))$ with $B(Q_l(\mb T))=B(\Diamond_l(\mb T))$. It is a lattice model, is obtained from the two-edge-deleted hive by sink--source extensions and admissible mutations, and every intermediate exchange matrix has full row rank.
\end{proposition}

\begin{proof}
Appendix~\ref{A:hive-combinatorics} gives an explicit deletion order for the alternating strip. At each step the vertex to be deleted is made a sink by mutating only vertices that remain, and a triangular minor proves full row rank before and after every deletion. Reversing the procedure and applying Corollary~\ref{C:iterated-fullrank} gives a rigid Jacobi-finite lattice model; Lemma~\ref{L:rigid-extension} gives admissibility. The deleted exterior frozen vertices are restored by frozen-source extensions using the boundary words in Lemma~\ref{L:hive-boundary-rays}. Twists and flips then give the prescribed consistent triangulation.

\end{proof}

\begin{proposition}\label{P:diagonal-localized-hive}
Let $E(\mb T)$ be the set of vertices on the internal edges. We have
\[ \uca(\Diamond_l(\mb T))[x_e^{-1}\mid e\in E(\mb T)] = \uca(\Diamond_l^{\mathrm{fr}}(\mb T))[x_e^{-1}\mid e\in E(\mb T)]. \]
This algebra is the iterated tensor product of the one-triangle upper cluster algebras, localized at their internal sides, over the Laurent rings of the shared edges. Products of their relative generic basis elements form a basis. For a potential satisfying Proposition~\ref{P:frozen-hive-amalgamation}, this is the partially localized generic basis in Corollary~\ref{C:relative-amalgamation}.
\end{proposition}

\begin{proof}
Lemma~\ref{L:glued-hive-rank} gives full row rank, so Proposition~\ref{P:freezing-localization} applies to $E(\mb T)$. Once the internal edges are frozen, mutations in different triangle interiors are independent. The partially localized Laurent rings are therefore tensor products over the Laurent rings of the shared coefficients. The one-triangle generic bases make the factor algebras free over these rings by Proposition~\ref{P:partial-localization-basis}. Apply Lemma~\ref{L:free-intersection} twice at each edge to obtain the tensor-product identity and its basis. For a potential as in the statement, Corollary~\ref{C:relative-amalgamation} identifies this basis with the generic characters of the glued IQP.
\end{proof}

\begin{corollary}\label{C:polyhedral-hive}
$(Q_l(\mb T),W_l(\mb T))$ is a polyhedral cluster model for $\uca(\Diamond_l(\mb T))$, with basis indexed by
\[ \{\g\in\mathbb Z^{\Diamond_l(\mb T)_0} \mid \vartheta_v(\g)\ge0 \text{ for every boundary vertex }v\}. \]
\end{corollary}

\begin{proof}
The boundary words of Lemma~\ref{L:hive-boundary-rays}, transported through the twists and flips, give relative source optimizations of all boundary vertices. Apply Proposition~\ref{P:hive-sequence}, Theorem~\ref{T:boundary-compactification}, and Proposition~\ref{P:source-polyhedral-support}.
\end{proof}

\subsection{Triple flags}\label{S:SI}

Let $Q$ be a finite acyclic quiver and $\beta$ a dimension vector. We write
\[ \SI_\beta(Q)=k[\Rep_\beta(Q)]^{\SL_\beta}. \]
Since $\Rep_\beta(Q)$ is an affine space and $\SL_\beta$ is connected with trivial character group, $\SI_\beta(Q)$ is factorial by Popov--Vinberg \cite[Theorem~3.17]{PV}.
For a projective presentation $f:P_1\to P_0$ whose induced Hom map is square on $\Rep_\beta(Q)$, let $s(f)$ be its Schofield determinant. Such determinants span every weight space \cite{S1,SV,DW1,DZ}; an indivisible extremal weight gives an irreducible semi-invariant \cite[Lemma~1.8]{Fs1}.

Let $S_l^m$ be the complete $m$-tuple flag quiver with its standard dimension vector $\beta_l$. For distinct flags $a,b,c$ and $i+j+k=l$, denote by
$ s_{i,j,k}^{a,b,c}$
the Schofield determinant associated with $P_l\to P_i^a\oplus P_j^b\oplus P_k^c$. These functions label the hive seed on the triangle $[a,b,c]$.

Put $T_l=S_l^3$ and write $\b s_l(i,j,k)=s_{i,j,k}^{1,2,3}$. Let $D_\partial$ be the product of all frozen hive variables, and let $D_c$ be the product on a boundary edge $c$. In this subsection $W_l$ denotes the standard hive potential of \cite{Fs1}.

\begin{theorem}[{\cite{Fs1}}]\label{T:triple}
We have
\[ \SI_{\beta_l}(T_l)=\uca(\Delta_l,\b s_l), \]
and the one-triangle hive IQP is a polyhedral cluster model.
\end{theorem}

\begin{proof}
The ring equality, generic-character basis, and hive-cone description are \cite[Theorem~8.11, Corollary~8.12, and Theorem~6.8]{Fs1}, respectively.
\end{proof}

\begin{proposition}\label{P:localized-triple}
For every product $D$ of frozen hive variables, we have
\[ \SI_{\beta_l}(T_l)_D=\uca(\Delta_l,\b s_l)_D, \]
with the partially localized generic basis of Proposition~\ref{P:partial-localization-basis}. In particular, for every boundary edge $c$,
\[ \SI_{\beta_l}(T_l)_{D_c}=\uca(\Delta_l,\b s_l)_{D_c}, \]
and
\[ \SI_{\beta_l}(T_l)_{D_\partial} =\ucaloc(\Delta_l,\b s_l) =\bigoplus_{\g\in\mathbb Z^{(\Delta_l)_0}} k\,C_{W_l}^{\mathrm{loc}}(\g). \]
\end{proposition}

\begin{proof}
Localize Theorem~\ref{T:triple} at the chosen coefficients and apply Proposition~\ref{P:partial-localization-basis}. For $D=D_\partial$, use Lemma~\ref{L:boundary-intersection} and~\eqref{eq:full-coefficient-saturation}.
\end{proof}

\subsection{Complete \texorpdfstring{$m$}{m}-tuple flags}
Let $\mb T$ be a consistent triangulation of $\mc D_m$. Glue the triple-flag determinants along the internal diagonals and write
\[ A_{\mb T}=\uca(\Diamond_l(\mb T),\b s_l(\mb T)), \qquad R_m=\SI_{\beta_l}(S_l^m). \]
Let $D(\mb T)$ be the product of the diagonal functions.
The next proposition will be proved in Appendix~\ref{A:flag-gluing}.

\begin{proposition}\label{P:localized-gluing}
For every consistent triangulation $\mb T$, we have
\[ (A_{\mb T})_{D(\mb T)}=(R_m)_{D(\mb T)} \]
inside their common fraction field. The tensor product of the factor partially localized generic-character bases gives a basis of this common algebra, via Proposition~\ref{P:diagonal-localized-hive}.
\end{proposition}

\begin{proposition}\label{P:flip-compatibility}
If $\mb T'$ is obtained from $\mb T$ by flipping an internal diagonal $d$, then the Fock--Goncharov mutation sequence $\bs\mu^d$ carries the normalized seed for $\mb T$ to that for $\mb T'$. Consequently $A_{\mb T}=A_{\mb T'}$.
\end{proposition}

This is the Fock--Goncharov flip sequence after identifying the Schofield determinants with their $\mathcal A$-coordinates; see \cite[Sections~10.2--10.3, especially (10.11)]{FG}.

\begin{proposition}\label{P:flip-cover}
Let $\mb T^\dagger$ be a consistent triangulation sharing no diagonal with $\mb T$. Then
\[ A_{\mb T} = (A_{\mb T})_{D(\mb T)} \cap (A_{\mb T})_{D(\mb T^\dagger)} \]
inside the common fraction field, where Proposition~\ref{P:flip-compatibility} identifies the two seeds in the same mutation class.
\end{proposition}

\begin{proof}
Let $S_\partial$ be the multiplicative system generated by the frozen boundary variables and put
\[ \widetilde A_{\mb T}=S_\partial^{-1}A_{\mb T} =\ucaloc(\Diamond_l(\mb T)). \]
Lemma~\ref{L:glued-hive-rank} gives full row rank and primitivity: in our row convention every mutable row contains an entry $\pm1$. Hence $\widetilde A_{\mb T}$ is factorial by \cite[Theorem~4.13(i)]{CKQ}.

The factors of $D(\mb T)$ and $D(\mb T^\dagger)$ are cluster variables. Distinct diagonal functions have distinct multidegrees, so they are distinct cluster variables. By \cite[Lemma~4.22]{CKQ}, they are nonassociate in $\widetilde A_{\mb T}$. Since the triangulations share no diagonal, the two products are coprime. Thus
\[ \widetilde A_{\mb T} = (\widetilde A_{\mb T})_{D(\mb T)} \cap (\widetilde A_{\mb T})_{D(\mb T^\dagger)}. \]

Now let $z$ belong to the intersection in the statement. The preceding equality gives $z\in\widetilde A_{\mb T}$. Every factor of either diagonal product is mutable, hence has valuation zero along every frozen boundary divisor by Lemma~\ref{L:frozen-prime}. Writing $z$ in either localization therefore shows that all boundary valuations of $z$ are nonnegative. Lemma~\ref{L:boundary-intersection} gives $z\in A_{\mb T}$. The reverse inclusion is immediate.
\end{proof}

\begin{theorem}[$m$-tuple equality]\label{T:mtuple}
For every consistent triangulation $\mb T$,
\[ \SI_{\beta_l}(S_l^m) = \uca(\Diamond_l(\mb T),\b s_l(\mb T)). \]
The generic characters of $(Q_l(\mb T),W_l(\mb T))$ with nonnegative boundary valuations form a polyhedral basis.
\end{theorem}

\begin{proof}
For $m=3$ the ring equality is Theorem~\ref{T:triple}. Assume $m\ge4$, and choose $\mb T^\dagger$ to be the one-step cyclic rotation of $\mb T$. A diagonal crosses its one-step rotation; if both belonged to $\mb T$, this would contradict noncrossing. Thus $\mb T$ and $\mb T^\dagger$ share no diagonal. Proposition~\ref{P:localized-gluing} gives the two localized equalities.

Each diagonal determinant is irreducible in $R_m$. For a triple containing its endpoints, it is a frozen prime by Theorem~\ref{T:triple} and Lemma~\ref{L:frozen-prime}. Every polynomial factor of this invariant is invariant, since the acting product of special linear groups is connected and has no nontrivial characters. The determinant is therefore irreducible in the coordinate ring of the triple-flag representation space. It remains irreducible after adjoining the matrix entries of the other flag arms, and hence is irreducible in $R_m$. Distinct diagonal determinants have distinct flag multidegrees and are nonassociate. The two triangulations share no diagonal, so their products have no common irreducible factor. Since the invariant ring is factorial by the theorem of Popov--Vinberg, we obtain
\[ R_m=(R_m)_{D(\mb T)}\cap(R_m)_{D(\mb T^\dagger)}. \]
Proposition~\ref{P:flip-cover} gives the corresponding intersection for the upper cluster algebra. Comparing the two intersections proves the ring equality. Corollary~\ref{C:polyhedral-hive} gives the basis statement. For the standard one-triangle potential, this is already part of Theorem~\ref{T:triple}.
\end{proof}

\appendix
\section{Hive combinatorics and extension sequences}\label{A:hive-combinatorics}

For an oriented triangle $\mathsf t=[a,b,c]$, relabel the vertices $(i,j,k)$ of the reference hive as
\[ \begin{bmatrix}a&b&c\\ i&j&k\end{bmatrix}. \]
The reference orientation is shown in Figure~\ref{f:triangle}.
\begin{figure}[H]
\centering
$\mtwo{1}{2}{3}{$\circlearrowleft$}{}{}{}$
\caption{The oriented reference triangle.}
\label{f:triangle}
\end{figure}

Two hives are glued by identifying the frozen vertices on the common edge and then adding the standard cross-edge arrows. The two local orientation patterns are shown in Figures~\ref{f:pair} and~\ref{f:Pair}.
\begin{figure}[H]
\centering
$\mthree{a}{b}{c}{d}{$\circlearrowleft$}{$\circlearrowleft$}{}{}\qquad
 \mthree{a}{b}{c}{d}{$\circlearrowleft$}{$\circlearrowright$}{}{}$
\caption{The two local orientation patterns for adjacent triangles.}
\label{f:pair}
\end{figure}
\begin{figure}[H]
\centering
\resizebox{\textwidth}{!}{$\diamondfive{\color{blue}}{\color{green}}\qquad\diamondfiveanti$}
\caption{Gluing two size-$5$ hives; the additional cross-edge arrows are red.}
\label{f:Pair}
\end{figure}

A triangulation is \emph{consistent} when adjacent triangles induce opposite orientations on their shared edge. A flip replaces the diagonal of a quadrilateral, as in Figure~\ref{f:flip}. The Fock--Goncharov sequence mutates the successive central rectangles in the order given in \cite[Section~10.3]{FG}.
\begin{figure}[H]
\centering
$\mthree{a}{b}{c}{d}{$\circlearrowleft$}{$\circlearrowleft$}{}{}\qquad
 \mthreeflip{a}{b}{c}{d}{$\circlearrowleft$}{$\circlearrowleft$}$
\caption{A flip of an internal diagonal.}
\label{f:flip}
\end{figure}

Consider the following hive strip of length $n+1$. The vertex $0$ can be frozen or mutable.
	\[ \hiverow{_0}{_1}{_{n-1}}{_n} \]
\begin{lemma}\label{L:hiverow}
The sequence of mutations $(n,n-1,\dots,1)$ transforms the quiver into
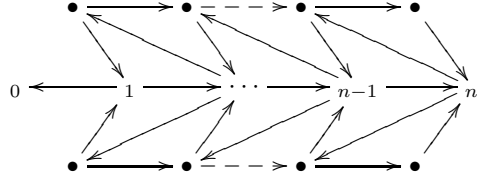
\[ \hiverowmuhalf{_0}{_1}{_{n-1}}{_{n}}. \]
In particular, it optimizes the vertex $0$. If we then delete $0$ and apply the sequence $(1,\dots,n-1,n)$, the quiver becomes the original one with $0$ deleted.
\end{lemma}

\begin{proof}
The assertion follows by induction on $n$; a related mutation sequence appears in \cite[Proposition~19]{Ma}.
\end{proof}

The deleted-edge models are given in \cite[Examples~5.10 and~5.12]{Fs2}. The following lemmas establish full row rank under gluing.
\begin{lemma} \label{L:gluerank} Let $\Diamond$ be obtained from two ice quivers $\Delta$ and $\Delta'$ by gluing along a set of frozen vertices $c$ (and unfreezing $c$).
Let $a$ and $a'$ (resp. $b$ and $b'$) be the remaining frozen vertices (resp. mutable) in $\Delta$ and $\Delta'$.
Let $\Delta^c$ be obtained from $\Delta$ by unfreezing $c$. Suppose that the submatrix of $B(\Delta^c)$ consisting of $b\cup c$-rows and $a\cup b$-columns has full row rank, and that $B(\Delta')$ has full row rank.
Then $B(\Diamond)$ has full row rank.
\end{lemma}

\begin{proof}
Consider a linear relation among the rows of $B(\Diamond)$, with coefficient vectors $\lambda_b,\lambda_c,\lambda_{b'}$. Restrict it to the columns $a\cup b$. The $b'$-rows vanish there, while the $(b\cup c)\times(a\cup b)$ block is the assumed full-row-rank submatrix of $B(\Delta^c)$. Hence $\lambda_b=\lambda_c=0$. The remaining relation, restricted to the columns $c\cup b'\cup a'$, is a relation among the rows of $B(\Delta')$. Its full row rank gives $\lambda_{b'}=0$. Thus all rows are independent.
\end{proof}

\begin{lemma}\label{L:fullrank}
Let $\Delta_l^{\flat'}$ be the triangular grid with one edge deleted, one remaining edge $a$ frozen, and the other remaining edge $c$ mutable. Retain all arrows with a mutable endpoint, including the corner arrow between $a$ and $c$. Let $b$ be the set of interior vertices. The submatrix on rows $b\cup c$ and columns $a\cup b$ is unimodular.
\end{lemma}

\begin{proof}
Relabel the barycentric coordinates so that the deleted side is $j=0$, the frozen side $a$ is $k=0$, and $c$ is $i=0$. The row vertices are
\[ R=\{(i,j,k):i+j+k=l,\ i\ge0,\ j,k\ge1\}. \]
Pair the row $r=(i,j,k)$ with the column $\phi(r)=(i+1,j,k-1)$. This is a bijection onto $a\cup b$. Order the rows lexicographically by $(i,j)$ and order the columns by the paired rows. The entry at $(r,\phi(r))$ is $\pm1$. If another column $\phi(s)$ is adjacent to the row vertex, the six possible differences between neighboring grid vertices give either $s_i<i$, or $s_i=i$ and $s_j<j$. Thus the matrix is lower triangular with diagonal entries $\pm1$. This uses the corner arrow, also when $l=2$.
\end{proof}

\begin{lemma}\label{L:glued-hive-rank}
For every consistent triangulation $\mb T$ of a disk, $B(\Diamond_l(\mb T))$ has full row rank. Each mutable row is primitive.
\end{lemma}

\begin{proof}
Induct by removing a leaf triangle of the dual tree. The one-triangle case is \cite[Lemma~6.1]{Fs1}. When the leaf is restored, let $c$ be its shared side, $b$ its interior, and $a$ its two exterior sides. In the $(b\cup c)\times(a\cup b)$ block, discard the columns on one exterior side. The remaining square block is exactly that of Lemma~\ref{L:fullrank}; all arrows incident with the now mutable side $c$ are retained by the gluing convention. Lemma~\ref{L:gluerank} and induction prove full row rank. In a consistent gluing, arrows along a shared side cancel in opposite pairs; all other grid arrows occur once. Thus all exchange entries are $0,1,-1$. Full row rank excludes a zero mutable row, proving primitivity.
\end{proof}

For a triangle $\t\in\mb T$ and an edge $\e$ of $\t$, define a \emph{twist} by the following steps. \begin{enumerate}
	\item Cut along the other two edges of the chosen triangle (see Figure \ref{f:twist} left);
	\item Change the orientation of the triangle and the identification of the two edges
	(the new identification is indicated by arrows, see Figure \ref{f:twist} right);
	\item Glue according to the new identification.
\end{enumerate}
\begin{figure}[!h] $\mtwod{>}{}{} \hspace{-.3in} \mthree{}{}{}{}{}{$\circlearrowleft$}{>}{>>}\hspace{-.3in} \mtwod{}{>>}{}
   \hspace{.3in} \mtwod{>}{}{} \hspace{-.3in} \mthree{}{}{}{}{}{$\circlearrowright$}{>>}{>}\hspace{-.3in} \mtwod{}{>>}{}$ \caption{Twisting a triangle along the chosen edge.} \label{f:twist}\end{figure}
Boundary edges are left uncut, and the reattachments are those indicated in Figure~\ref{f:twist}. Denote the resulting triangulation by $\mb T^{\t,\e}$. The operation preserves the disks considered here; for general surfaces see \cite{Fk2}.

\begin{lemma} \label{L:twist} There is a sequence of mutations $\bs{\mu}^{\t,\e}$ such that
	\[ \bs{\mu}^{\t,\e}(\Diamond_l(\mb{T}))=\Diamond_l(\mb{T}^{\t,\e}). \]
\end{lemma}

\begin{proof}
The mutation sequence and the calculation of the interior arrows are given in \cite[Appendix~B.2 and Corollary~B.9]{Fart}. The arrows on the chosen edge are reversed by \cite[Lemma~9]{Ma}; the other subsequences do not affect that edge. The remaining choices of edge follow by cyclic symmetry.
\end{proof}

\subsection{The row-deletion construction}

Put $t=m-2$ and use lattice coordinates $(a,b)$ on the alternating strip, with arrow directions $(1,0)$, $(0,-1)$, and $(-1,1)$. The strip has $0\le b\le l$ and $0\le a\le rl$ for $t=2r$, or $0\le a\le(r+1)l-b$ for $t=2r+1$; omit the triangle corners. First delete all exterior frozen vertices except $F_0=\{(0,b):1\le b\le l-1\}$. The remaining mutable vertices are
\[ M=\{(a,b):1\le b\le l-1,\ 1\le a\le R_b\}, \]
where $R_b=rl-1$ for $t=2r$ and $R_b=(r+1)l-b-1$ for $t=2r+1$.

Delete the mutable vertices with $a+b\ge l$ in decreasing order of $b$, and within each row in decreasing order of $a$. Before deleting $e=(a,b)$, mutate
\[ (a,1),(a,2),\ldots,(a,b-1). \]
All vertices in the word remain: the lower rows have not yet been shortened and $R_j\ge R_b$ for $j<b$. Put $p_j=(a,b-j)$, $0\le j\le b-1$. An outside vertex $w$ is joined to the path by an oriented triangle $p_{j-1}\to p_j\to w\to p_{j-1}$, by the single arrow $w\to p_{b-1}$, or not at all. The possible unpaired outgoing neighbors $(a+1,b)$, $(a,b+1)$, and $(a-1,b+1)$ are absent. The paired incidences cancel successively; the single incoming incidence at the far endpoint creates no outgoing arrow at $p_0$. Thus $e$ is a sink, also when $b=1$ and the word is empty. Deleting $e$ and undoing the word leaves the original grid with $e$ removed. The remaining vertices in each horizontal row form an initial interval.

For every such intermediate quiver, pair the mutable row $(a,b)$ of the exchange matrix with the column $(a-1,b)$. Ordering rows lexicographically by $(a,b)$ and columns by their paired rows gives a triangular square submatrix with diagonal entries $\pm1$. Hence every intermediate exchange matrix has full row rank. The terminal quiver is $\Delta_l^{\flat\flat}$. Reversing the deletions gives the sink--source extension sequence used in Proposition~\ref{P:hive-sequence}.

\begin{lemma}\label{L:hive-boundary-rays}
Every exterior frozen vertex of the alternating strip admits a relative source optimization.
\end{lemma}

\begin{proof}
For $v=(a,b)$, take $h=(1,-1)$ if $a=0$ or $b=l$, $h=(0,1)$ if $b=0$, and $h=(-1,0)$ on the remaining right boundary, using the first applicable case. Let $p_0=v$ and let $p_j=v+jh$, $1\le j\le n$, be the consecutive mutable vertices on this ray. After deleting the other frozen vertices, the path is $p_n\to\cdots\to p_1\to p_0$. An outside vertex $w$ has the two incidences in a triangle $p_j\to p_{j-1}\to w\to p_j$, the single outgoing incidence $p_n\to w$, or no incidence. Mutating $p_n,\ldots,p_1$ gives the same cancellation with arrows reversed, so $v$ becomes a source. If $n=0$, it is already a source.
\end{proof}

To restore the deleted exterior frozen vertices, apply the corresponding boundary word, adjoin the new frozen vertex as an actual source with the required incidences to mutable vertices and no arrows to the existing frozen vertices, and undo the word. Mutation of a coefficient column depends only on the mutable principal part and that column, so the prescribed exchange matrix is recovered. Adding frozen columns preserves full row rank. The frozen-extension theorem and Lemma~\ref{L:rigid-extension} therefore preserve the lattice-model, rigidity, and Jacobi-finiteness properties. Finally apply the twists of Lemma~\ref{L:twist} and the standard flip words to reach any consistent triangulation.

\section{Gluing of generic pairs}\label{A:flag-gluing}

Let $G=\SL_l$, and let $\mc A$ be the affine decorated-flag variety obtained from one complete flag arm by quotienting its internal base-change group. Its open $G$-orbit is $G/U^-$, where $U^-$ is a maximal unipotent subgroup, and
\[ R_m=k[\mc A^m]^G. \]
Let $F^-,F^+\in\mc A$ be the standard opposite decorated flags, with stabilizers $U^-$ and $U^+$, and let $H\subset G$ be the diagonal torus.  Denote by $P^\circ\subset\mc A\times\mc A$ the open set on which all two-flag determinants are nonzero.

\begin{lemma}[{\cite[Section~5]{FG}}]\label{L:generic-pair-slice}
The map
\[ G\times H\longrightarrow P^\circ, \qquad (g,h)\longmapsto(gF^-,ghF^+), \]
is an isomorphism.  The diagonal $G$-action corresponds to left multiplication on the first factor. In particular, the action on $P^\circ$ is free and its geometric quotient is $H$.
\end{lemma}

Let $\mc D(\mb T)$ be the set of internal diagonals. For each $d\in\mc D(\mb T)$, fix an ordering of its two endpoint flags. If a triangle $\mathsf t$ contains $d$, forgetting the third flag and applying Lemma~\ref{L:generic-pair-slice} gives a morphism
$ q_{\mathsf t,d}:X_3^{\mathsf t,\circ}\longrightarrow H_d$,
where $H_d$ is a copy of $H$ and $X_3^{\mathsf t,\circ}$ is the open triple-flag quotient on which every internal side of $\mathsf t$ is generic. Use the same ordering of $d$ for both adjacent triangles.

\begin{proposition}\label{P:generic-pair-fiber-product}
Let $ X_m^\circ = \Spec (R_m)_{D(\mb T)}$.
Then the triangle projections induce an isomorphism of affine varieties
\begin{equation}\label{eq:flag-fiber-product} X_m^\circ \simeq \mathop{\times}_{d\in\mc D(\mb T),\,H_d} X_3^{\mathsf t,\circ}, \end{equation}
where the right side denotes the iterated fiber product imposing equality of the two $H_d$-parameters along every internal diagonal. Consequently,
\begin{equation}\label{eq:flag-tensor-product} (R_m)_{D(\mb T)} \simeq \mathop{\bigotimes}_{d\in\mc D(\mb T),\,k[H_d]} k[X_3^{\mathsf t,\circ}]. \end{equation}
\end{proposition}

\begin{proof}
The dual graph of a disk triangulation is a tree, so we may add one leaf triangle at a time. Suppose a leaf triangle $\mathsf t$ meets the previous polygon along the ordered diagonal $d=(b,c)$ and adds a new flag $a$. By Lemma~\ref{L:generic-pair-slice}, the generic ordered pair $(A_b,A_c)$ has a unique normalization to $(F^-,hF^+)$, depending regularly on the pair. The third flag then gives a point of $\mc A$, and $h\in H_d$ is the common parameter along the diagonal. A point of the fiber product therefore determines a unique $G$-orbit of the enlarged tuple, and this correspondence is regular in both directions. Induction proves~\eqref{eq:flag-fiber-product}. All spaces are affine principal open subsets of affine reductive quotients, so taking coordinate rings gives~\eqref{eq:flag-tensor-product}.
\end{proof}

\begin{proof}[Proof of Proposition~\ref{P:localized-gluing}]
For $m=3$, there are no internal diagonals and $D(\mb T)=1$, so the claim is Theorem~\ref{T:triple}. Assume $m\ge4$. For each triangle $\mathsf t$, let $D_{\mathsf t}$ be the product of the functions on its internal sides. Proposition~\ref{P:localized-triple} gives
\[ k[X_3^{\mathsf t,\circ}] =\SI_{\beta_l}(T_l)_{D_{\mathsf t}} =\uca(\Delta_l^{\mathsf t},\b s_l^{\mathsf t})_{D_{\mathsf t}}. \]
Substitute these equalities into~\eqref{eq:flag-tensor-product} and apply Proposition~\ref{P:diagonal-localized-hive}. Each factor cluster variable maps to the corresponding Schofield determinant, so the equality holds in the common field of rational functions. Products of the factor basis elements give the asserted basis.
\end{proof}

\section*{Acknowledgement}
The first author thanks the National Center for Theoretical Sciences and the University of Connecticut for financial support.

\end{document}